\documentclass[11pt]{amsart}
\usepackage{amssymb,latexsym,amsmath}
\usepackage[mathscr]{eucal}

\numberwithin{equation}{section}

\newtheorem{corollary}[equation]{Corollary}
\newtheorem*{corollary*}{Corollary}
\newtheorem{lemma}[equation]{Lemma}
\newtheorem*{lemma*}{Lemma}
\newtheorem{proposition}[equation]{Proposition}
\newtheorem*{proposition*}{Proposition}
\newtheorem{theorem}[equation]{Theorem}
\newtheorem*{theorem*}{Theorem}

\theoremstyle{remark}
\newtheorem*{assume*}{Assume}

\newtheorem*{claim*}{Claim}

\newtheorem{definition}[equation]{Definition}
\newtheorem*{definition*}{Definition}
\newtheorem{example}[equation]{Example}
\newtheorem*{example*}{Example}
\newtheorem*{hint*}{Hint}
\newtheorem*{notation*}{Notation}
\newtheorem*{question*}{Question}
\newtheorem*{answer*}{Answer}
\newtheorem{remark}[equation]{Remark}
\newtheorem*{remark*}{Remark}

\numberwithin{HWeq}{section}
\theoremstyle{definition}

\def\a{\alpha}
\def\b{\beta}
\def\c{\gamma}
\def\d{\delta}

\def\e{\varepsilon}

\def\z{\zeta}
\def\m{\mu}
\def\n{\nu}
\def\s{\sigma}
\def\t{\tau}
\def\O{\Omega}
\def\w{\omega}

\def\bara{{\bar\alpha}}

\def\barb{{\bar\beta}}

\def\barc{{\bar\gamma}}

\def\bare{{\bar\varepsilon}}

\def\barO{{\bar\Omega}}
\def\bC{\mathbb C}

\def\bP{\mathbb P}

\def\bR{\mathbb R}

\def\bZ{\mathbb Z}
\def\cF{\mathcal F}
\def\cG{\mathcal G}

\def\cL{\mathcal L}

\def\cP{\mathcal P}

\def\fg{{\mathfrak{g}}}

\def\fgl{\mathfrak{gl}}

\def\fsu{\mathfrak{su}}

\def\sfa{\mathsf{a}}

\def\sfb{\mathsf{b}}

\def\sfb{\mathsf{b}}

\def\sfc{\mathsf{c}}

\def\td{\mathrm{d}}

\def\sfe{\mathsf{e}}

\def\sff{\mathsf{f}}

\def\tGL{\mathrm{GL}}
\def\tGr{\mathrm{Gr}}

\def\ti{\mathrm{i}}

\def\sfi{\mathsf{i}}

\def\tId{\mathrm{Id}}

\def\tIm{\mathrm{Im}}

\def\sfJ{\mathsf{J}}

\def\tmod{\mathrm{mod}}

\def\tspan{\mathrm{span}}

\def\tSU{\mathrm{SU}}
\def\tSym{\mathrm{Sym}}

\def\sfw{\mathsf{w}}

\def\sfz{\mathsf{z}}
\def\del{\partial}

\def\barz{\bar z}

\def\ot{\otimes}
\def\wt{\widetilde}
\def\wh{\widehat}

\def\half{\tfrac{1}{2}}
\def\ihalf{\tfrac{\ti}{2}}

\def\ifourth{\tfrac{\sfi}{4}}
\def\fourth{\tfrac{1}{4}}

\newcounter{cnt}

\newcounter{exam_cnt}

\begin{document}
\title{Projective invariants of CR--hypersurfaces}
\author{C. Hammond}
\email{hammondc@math.tamu.edu}
\author[C. Robles]{C. Robles${}^\ast$}
\thanks{${}^\ast$Robles is partially supported by NSF-DMS 0805782 \& 1006353.}
\email{robles@math.tamu.edu}
\address{Mathematics Department, Mail-stop 3368, Texas A\&M University, College Station, TX  77843-3368}
\date{\today}
\begin{abstract}
We study the equivalence problem under projective transformation for CR-hypersurfaces of complex projective space.  A complete set of projective differential invariants for analytic hypersurfaces is given.  The self-dual strongly $\bC$-linearly convex hypersurfaces are characterized.  
\end{abstract}
\maketitle

%------------------------------------------------------------------------------
%------------------------------------------------------------------------------

%------------------------------------------------------------------------------
\section{Introduction}
%------------------------------------------------------------------------------

%------------------------------------------------------------------------------
\subsection{History}
%------------------------------------------------------------------------------

Before stating the equivalence problem addressed in this paper, we describe two related problems previously addressed in the literature.

Cartan studied the equivalence of real-analytic CR-hypersurfaces in complex 2-manifolds up to local biholomorphism \cite{MR1553196, MR1556687}.  Chern and Moser generalized Cartan's result \cite{MR0425155, MR709144}, solving the equivalence problem for nondegenerate, real-analytic CR-hypersurfaces in complex $n$-manifolds modulo local biholomorphism \cite[Theorem 4.6]{MR0425155}.  

Following Chern and Moser, Jensen considered the equivalence problem for nondegenerate CR-hypersurfaces of complex projective space $\bP W$ up to (local) projective deformation  \cite{MR0500648, MR836926, MR739891}.  He showed that two nondegenerate CR-hypersurfaces in $\bP W$ are locally first-order projective deformations of each other if and only if they locally biholomorphically equivalent \cite[Theorem 7.2]{MR739891}, and that two smooth CR--hypersurfaces are locally projectively equivalent if and only if they are locally second-order projective deformations of each other \cite{MR836926}.

%------------------------------------------------------------------------------
\subsection{Statement of the problem}
%------------------------------------------------------------------------------

In this paper we apply techniques of E. Cartan \cite{MR1190006, MR0410607} to characterize CR-hypersurfaces in complex projective space up to projective transformation.   While a projective transformation is a local biholomorphism of complex projective space $\bP W$, it is not the case that every local biholomorphism of $\bP W$ is a projective transformation:  in particular we consider the equivalence problem under a smaller transformation group than Cartan and Chern--Moser.  Likewise, Jensen's projective deformation is a more also a more general equivalence than projective congruence.  (See Remark \ref{R:Jensen}.)
\smallskip

Let $W$ be a complex vector space, let $\bP W$ denote the associated complex projective space.  Given $w \in W$, let $[w]$ denote the corresponding point in $\bP W$.  Let $\tGL(W)$ denote the space of complex linear automorphisms of $W$.

\begin{definition} \label{D:proj}
Every $A \in \tGL(W)$ naturally induces a \emph{projective linear transformation} $A : \bP W \to \bP W$ by $[w] \mapsto [Aw]$.  Two submanifolds $S,\, \wt S \in \bP W$ are \emph{projectively equivalent} if there exists a projective linear transformation $A$ such that $A(S) = \wt  S$.
\end{definition}

%------------------------------------------------------------------------------
\subsection{Contents}
%------------------------------------------------------------------------------

The equivalence problem is lifted from CR-hypersurfaces $S \subset \bP W$ to frame bundles $\cF_S \subset  \tGL(W)$ over $S$ in Section \ref{S:F}.  Projective differential invariants (of all orders) are defined in Section \ref{S:h}, and in Section \ref{S:local_comp_w} it is shown that set of invariants is complete (in the analytic category): they characterize the analytic CR-hypersurfaces up to projective equivalence (Proposition \ref{P:hcomplete}).  Modulo normalizations (in Section \ref{S:P_S}), the set of second-order projective invariants is precisely the CR--second fundamental form.  This is illustrated in Section \ref{S:wrtO} which presents the invariants from a slightly different, but sometimes computationally more convenient, perspective.  

The self-dual strongly $\bC$--linearly convex hypersurfaces are characterized in Section \ref{S:SD}, see Theorem \ref{T:SD}.  

%------------------------------------------------------------------------------
\subsection{Notation}
%------------------------------------------------------------------------------
Fix the index ranges
$$
\begin{array}{rclrclrcl}
  0 & \le & j,k \ \le \ n+1 \,, \qquad &
  0 & \le & a,b \ \le \ m \,, \qquad & 
  1 & \le & s,t \ \le \ n \\
  & & &
  0 & < & \a,\b \ < \ m \,, &
  1 & < & \sigma , \tau \ < \ n\, .
\end{array}
$$

Let $V$ be a real vector space of dimension $n+2 = 2m+2$ with a complex structure $\sfJ : V \to V$, $\sfJ^2 = -\tId$.  Give $(V,\sfJ)$ the structure of a complex vector space $W$ by defining $\ti v := \sfJ v$, where $\ti = \sqrt{-1}$.  Define an equivalence relation on $V\backslash\{0\}$ by $v \sim (x v + y \sfJ v)$ for any $(0,0)\not=(x,y) \in\bR^2$.  Then the quotient space is naturally identified with $\bP W \simeq \bC\bP^m$. 

Fix a basis $e = (e_0,\ldots,e_{n+1})$ of $V$ with the property that 
\begin{equation} \label{E:e}
  \sfJ \, e_{2a} \ = \ e_{2a+1} \quad \hbox{ and } \quad
  \sfJ \, e_{2a+1} \ = \ -e_{2a} \quad \forall \ 0 \le a \le m \, .
\end{equation}
The 
$$
  f_a \ := \ e_{2a}\, , \quad 0 \le a \le m \, , 
  %\half ( e_{2a} \,-\, \ti\, e_{2a+1} ) 
$$
form a basis of of the complex vector space $W$.  

\begin{remark} \label{R:evf}
At times we will find it most convenient to work with the frame $e$, while at other times the frame $f$ is better suited to the computation at hand.  For this reason we will develop both perspectives in the sections below.  See also Remark \ref{R:wvO}. 
\end{remark}

Define $\fgl(V,\sfJ) := \{ X \in \fgl(V) \ | \ X \sfJ = \sfJ X\} \simeq \
\fgl(W)$.  Let $X \in \fgl(V)$ be given by a matrix $(X^j_k)$ with respect to the basis $(e_0,\ldots ,e_{n+1})$.  Then $X \in \fgl(V,\sfJ)$ if and only if 
\begin{equation} \label{E:mA}
  X^{2a}_{2b} \ = \ X^{2a+1}_{2b+1} \quad \hbox{ and } \quad 
  X^{2a}_{2b+1} \ = \ -X^{2a+1}_{2b} \, .
\end{equation}
In this case, as an element of $\fgl(W)$, $X$ is given by the matrix $(Y^a_b)$, 
\begin{equation} \label{E:RvC}
  Y^a_b \ = \ X^{2a}_{2b} + \ti X^{2a+1}_{2b}
\end{equation}
with respect to the basis $(f_0 , \ldots , f_m)$ of $W$.

%------------------------------------------------------------------------------
\section{Frame bundles}
%------------------------------------------------------------------------------

%------------------------------------------------------------------------------
\subsection{Frame bundle over \boldmath $\bP W$ \unboldmath} \label{S:F}
%------------------------------------------------------------------------------

Let $\cF$ denote the set of frames (or bases) $e = (e_0,\ldots,e_{n+1})$ of $V$ satisfying \eqref{E:e}.  If we fix a frame, then $\cF$ may be identified with $\tGL(V,\sfJ) \simeq \tGL(W)$.  Given $A \in \tGL(W)$, let $L_A : \tGL(W) \to \tGL(W)$ denote left-multiplication by $A$.  Let $\pi : \tGL(W) \to \bP W$ denote the projection $\pi(e) = [e_0]$.  Then the diagram below commutes.
\begin{center}%\label{F:diagram}
\setlength{\unitlength}{0.8cm}
\begin{picture}(3.5,3.5)(0,0)
  \put(0.2,0){$\bP W$} 
  \put(4,0){$\bP W$}
  \put(1.6,0.1){\vector(1,0){1.8}}
  \put(2.3,0.3){$A$}
  \put(0,2.5){$\tGL(W)$}
  \put(3.8,2.5){$\tGL(W)$}
  \put(1.7,2.6){\vector(1,0){1.8}}
  \put(2.2,2.8){$L_A$}
  \put(0.5,2.2){\vector(0,-1){1.5}} \put(0,1.6){$\pi$}
  \put(4.3,2.2){\vector(0,-1){1.5}} \put(4.5,1.6){$\pi$}
\end{picture}\end{center}

%------------------------------------------------------------------------------
\subsubsection{The Maurer-Cartan form} \label{S:mc}
%------------------------------------------------------------------------------
The $\fgl(V,\sfJ)$--valued Maurer-Cartan 1-form $\w = (\w^j_k)$ on $\cF$ is defined by 
\begin{equation} \label{E:de}
  \td \, e_j \ = \ \w^k_j \, e_k \, .
\end{equation}
It is straight-forward to check that $\w$ is left-invariant:
\begin{equation} \label{E:Linvar}
  L_A^* \w \ = \ \w \, ,
\end{equation}
for all $A \in \tGL(W)$.  We also have the \emph{Maurer-Cartan equation}
\begin{equation} \label{E:mcR}
  \td \w^j_k \ = \ -\w^j_\ell \wedge \w^\ell_k \, ,
\end{equation}
and, by \eqref{E:e}, 
\begin{equation}\label{E:w}
  \w^{2a}_{2b} \ = \ \w^{2a+1}_{2b+1} \quad \hbox{ and } \quad  
  \w^{2a}_{2b+1} \ = \ -\w^{2a+1}_{2b} \, . 
\end{equation}

If
\begin{equation} \label{E:W=(w)}
  \Omega^b_a \ := \ \w^{2b}_{2a} \ + \ \ti \, \w^{2b+1}_{2a} \, , 
\end{equation}
then 
\begin{equation} \label{E:df}
  \td f_a \ = \ \Omega^b_a f_b
\end{equation}
and we have the \emph{Maurer-Cartan equation}
\begin{equation} \label{E:mcC}
\td \Omega^b_a \ = \ -\Omega^b_c \wedge \Omega^c_a \, .
\end{equation}
In particular, $\Omega = (\Omega^a_b)$ is the $\fgl(W)$--valued Maurer-Cartan form.  In analogy with \eqref{E:Linvar} we have
\begin{equation} \label{E:LinvarO}
  L_A^* \Omega \ = \ \Omega \, .
\end{equation}
For later convenience we note that 
\begin{equation} \label{E:w=(W)}
  \w^{2b}_{2a} \ = \ \half( \Omega^b_a + \bar\Omega^b_a ) \, , \qquad
  \w^{2b+1}_{2a} \ = \ -\ihalf ( \Omega^b_a - \bar\Omega^b_a ) \, .
\end{equation}

\begin{remark} \label{R:wvO}
As a follow up to Remark \ref{R:evf} we note that generally the $\fgl(W)$--valued Maurer-Cartan form $\Omega$ is more convenient in computations.  However, we found that the $\fgl(V,\sfJ)$--valued $\w$ has the advantage of yielding a general formula for the $k$-th order differential invariants (Proposition \ref{P:h}).  So we will continue to work with both, favoring the more convenient form for the computation at hand, and often giving equivalent statements for each.
\end{remark}

%------------------------------------------------------------------------------
\subsubsection{Change of frame} \label{S:tw}
%------------------------------------------------------------------------------
It will be necessary to understand how $\Omega$ varies under a change of frame.  Consider a smooth map $g = (g^a_b) : \cF \to \tGL_{m+1}\bC$.  Let $\tilde f_a = g_a^b(f) f_b$.  This map induces $G : \cF \to \cF$ by mapping the frame $f = (f_a)$ to the frame $G(f) = (g^b_a(f)f_b) = \tilde f$.  Let $\wt\O$ denote the pull-back $G^*(\Omega_{\tilde f})$.  Then \eqref{E:df} yields 
\begin{equation} \label{E:tO}
  \wt\O^b_a \ = \ 
  (g^{-1})^b_c \,\td g^c_a \ + \ (g^{-1})^b_c \, \Omega^c_d \, g^d_a \, .
\end{equation}

%------------------------------------------------------------------------------
\subsection{Frame bundle over a CR--hypersurface in \boldmath $\bP W$ \unboldmath}
\label{S:F_S}
%------------------------------------------------------------------------------

Let $S \subset \bP W$ be a CR--hypersurface.  Let $\wh S := \{ v \in V\backslash\{0\} \ | \ [v] \in S \}$ denote the cone over $S$.  Then $\wh S$ is a CR--hypersurface in $(V,\sfJ)$.  If $z = [e_0] \in S$, let $\wh T_z S \subset V$ be the $(n+1)$--plane tangent to $\wh S$ at $e_0$.  (We distinguish the linear subspace $\wh T_zS$ from the intrinsic tangent space $T_{e_0} \wh S$.)  Define 
$$
  \cF_S \ := \ \{ (e_0,\ldots,e_{n+1}) \in \cF \ | \ z = [e_0] \in S \, , \
  \tspan\{e_0,\ldots,e_n\} = \wh T_z S \} \, .
$$

\begin{remark} \label{R:LAFS}
Given $A \in \tGL(W)$, notice that $L_A(\cF_S) = \cF_{\wt S}$ where $\wt S = 
A(S) = \pi \circ L_A(\cF_S)$.  In particular, two CR-hypersurfaces $S , \wt S \subset \bP W$ are projectively equivalent if and only if there exists $A \in \tGL(W)$ such that $\cF_{\wt S} = L_A(\cF_S)$.
\end{remark}  

\noindent The key ingredient in the solution of the projective equivalence problem is the following.

\begin{proposition} \label{P:proj_equiv}
Two CR-hypersurface $S , \wt S \subset \bP W$ are projectively equivalent if and only if there exists a smooth map $\phi : \cF_S \to \cF_{\wt S}$ such that $\phi^*(\Omega_{|\cF_{\wt S}}) = \Omega_{|\cF_S}$.  (Equivalently, $\phi^*(\omega_{|\cF_{\wt S}}) = \omega_{|\cF_S}$.)
\end{proposition}

\noindent The proof of the proposition makes use of the following well-known theorem of E. Cartan.

\begin{theorem}[Cartan {\cite[Theorem 1.6.10]{MR2003610}}] \label{T:Cartan}
let $G$ be a Lie group with Lie algebra $\fg$ and Maurer-Cartan form $\vartheta$.  Let $M$ be a manifold with a $\fg$--valued 1-form $\eta$ satisfying the Maurer-Cartan equation $\td \eta = -[\eta,\eta]$.  Then for any $p \in M$ there exists a neighborhood $U$ and a map $f : U \to G$ such that $f^*\vartheta = \eta$.  Moreover any two such maps $f_1,f_2 : U \to G$ are related by $f_2 = L_g \circ f_1$ for some $g \in G$.
\end{theorem}

\begin{remark*}
In the case that $\fg$ is a matrix Lie algebra, $\td \eta = -[\eta,\eta]$ is \eqref{E:mcR}.
\end{remark*}

\begin{proof}[Proof of Proposition \ref{P:proj_equiv}]
The proposition follows directly from Theorem \ref{T:Cartan} by taking $G = \tGL(W)$ and $\vartheta = \Omega$ with $M = \cF_S$ and $\eta = \Omega_{|\cF_S}$, $f_1$ the inclusion map and $f_2 = \phi$.
\end{proof}

Consider the pull-back of the Maurer-Cartan form to $\cF_S$.  By construction $\{\td e_0(\xi) \ | \ \xi \in T_v\cF_S\} =  \tspan\{ e_0 , \ldots , e_n\}$ so that the $\w^0_0 , \w^1_0 , \ldots , \w^n_0$ are linearly independent on $\cF_S$, and 
\begin{equation}\label{E:w^n+1_0}
  \w^{n+1}_0 \ = \ 0 \qquad \stackrel{\eqref{E:w=(W)}}{\Longleftrightarrow} \qquad
  \Omega^m_0 \ - \ \bar\Omega^m_0 \ = \ 0 \, .
\end{equation}  

\begin{definition*}
A tangent vector $v \in T\cF_S$ is \emph{vertical} if $\pi_*(v) = 0$, where $\pi : \cF_S \to S$ is the natural projection.  A 1-form $\varrho$ on $\cF_S$ is \emph{semi-basic} if $\varrho(v) = 0$ for all vertical $v \in T\cF_S$.  
\end{definition*}

\begin{remark} \label{R:sb}
The equation $\td e_0 = \w^j_0 e_j$ implies that the linearly independent $\w^2_0 , \ldots , \w^n_0$ span the semi-basic 1-forms.  Equivalently, 1-forms $\Omega^\a_0$, $\bar\Omega^\a_0$ and $\Omega^m_0$ are linearly independent on $\cF_S$ and span the semi-basic 1-forms (as a real vector space).
\end{remark}

%------------------------------------------------------------------------------
\section{Projective differential invariants w.r.t $\w$}
%------------------------------------------------------------------------------

%------------------------------------------------------------------------------
\subsection{Definition and construction} \label{S:h}
%------------------------------------------------------------------------------

An application of the Maurer-Cartan equation \eqref{E:mcR} to \eqref{E:w^n+1_0} yields 
\begin{equation} \label{E:dw} \renewcommand{\arraystretch}{1.5}
\begin{array}{rcl}
  0 & = & -\td \w^{n+1}_0 \ = \ \w^{n+1}_1 \wedge \w^1_0 
  \ + \ \w^{n+1}_2 \wedge \w^2_0 \ + \ \cdots 
  \ + \ \w^{n+1}_n \wedge \w^n_0 \, .
\end{array}\end{equation}
Recall that the $\w^1_0,\ldots , \w^n_0$ are linearly independent on $\cF_S$, by Remark \ref{R:sb}.  The following lemma is well-known.

\begin{lemma}[Cartan's Lemma {\cite{MR1083148,MR2003610}} for $\w$] \label{L:Cartan4w}
Suppose there exist real-valued 1-forms $\xi_1 , \ldots , \xi_n$ on $\cF_S$ such that $\xi_s \wedge \w^s_0 = 0$.  Then there exist functions $H_{st} = H_{ts} : \cF_S \to \bR$, $1 \le s,t \le n$, such that 
$$
  \xi_s \ = \ H_{st}\,\w^t_0 \, .
$$
\end{lemma}

Lemma \ref{L:Cartan4w} and \eqref{E:dw} imply there exist functions $h_{st} = h_{ts} : \cF_S \to \bR$, $1 \le s,t \le n$, such that 
\begin{equation} \label{E:h2}
  \w^{n+1}_s \ = \ h_{st} \, \w^t_0 \, , \ 1 \le s \le n  \,.
\end{equation}
Note that $\w^n_0 \, \stackrel{\eqref{E:w}}{=} \, \w^{n+1}_1 \, = \, h_{1s} \, \w^s_0$.  This forces
\begin{equation} \label{E:h2_id}
  h_{1s} \ = \ h_{s1} \ = \ 0 \,, \ \forall \ s < n \,,\quad \hbox{ and } \quad
  h_{1n} \ = \ h_{n1} \ = \ 1 \, .
\end{equation}

The $h^{(2)} = \{h_{st}\}$ are second-order projective differential invariants: they describe the second-order differential geometry of $S$ at $z = [e_0]$;  see Section \ref{S:local_comp_w} for more detail.  The $\{h_{\sigma\tau}\} \subset h^{(2)}$ are the coefficients of the CR second fundamental form; see Section \ref{S:F2}.

We repeat the process above to obtain third-order invariants: apply the Maurer-Cartan equation \eqref{E:mcR} to differentiate \eqref{E:h2}.  An application of Cartan's Lemma \ref{L:Cartan4w} to the resulting 2-form yields functions $h_{rst} : \cF_S \to \bR$ that are fully symmetric in the indices such that 
\begin{equation} \label{E:h3}
  \td h_{rs} \ + \ h_{rs}\,(\w^0_0 + \w^{n+1}_{n+1}) \ - \ 
  h_{rt}\,\w^t_s \ - \ h_{ts}\,\w^t_r \ = \ 
  h_{rst}\,\w^t_0 \, .
\end{equation}

The together the coefficients $h^{(2)} = \{h_{st}\}$ and $h^{(3)} = \{h_{rst}\}$ describe the geometry of $S$ to third order.  See \S\ref{S:local_comp_w}.

Continuing inductively one may deduce a general formula (Proposition \ref{P:h}) for the $(p+1)$-st order invariants $h_{s_1\cdots s_p t}$.   First some notation:  given two tensors $T_{s_1 \ldots s_p}$ and $U_{s_{p+1} \ldots s_{p+q}}$, let $S_{p+q}$ denote the symmetric group on $p+q$ letters.  Let
$$
  T_{(s_1 \ldots s_p} U_{s_{p+1} \ldots s_{p+q})} \ = \ 
  \frac{1}{(p+q)!} \sum_{\sigma\in S_{p+q}} \, 
  T_{\sigma(a_1) \ldots \sigma(s_p)} U_{\sigma(s_{p+1}) \ldots \sigma(s_{p+q})}
$$ 
denote the \emph{symmetrization} of their product.  For example, $T_{(s_1} U_{s_2)} = \half( T_{s_1} U_{s_2} + T_{s_2} U_{s_1})$.  We exclude from the symmetrization operation any index that is outside the parentheses.  For example, in $h_{t(s_1 \ldots s_{p-1}} \, \w^t_{s_p)}$ we symmetrize over only the $s_i$, excluding the $t$ index. 

\begin{proposition} \label{P:h}
Set $h_s = 0$ and assume $p>1$.  There exist functions $h_{s_1\cdots s_p t} : \cF_S \to \bR$, fully symmetric in their indices, such that 
\begin{eqnarray*}
  h_{s_1 \ldots s_p t} \, \w^t_0 & = &
  \td h_{s_1 \ldots s_p} \, + \, (p-1) \, h_{s_1 \ldots s_p} \, \w^0_0
  \, + \, h_{s_1 \ldots s_p} \, \w^{n+1}_{n+1} \hfill \\
  & &
  + \ p \,  
  \left\{ (p-2) \, h_{(s_1 \ldots s_{p-1}} \w^0_{s_p)} \,  
       - \, h_{t(s_1 \ldots s_{p-1}} \, \w^t_{s_p)}  \right\} \\
  & & 
  + \ \sum_{j=1}^{p-2} \, \tbinom{p}{j} \, 
      \left\{
      (j-1)\, h_{(s_1 \ldots s_j} \, h_{s_{j+1} \ldots s_p)} \, \w^0_{n+1}
      \ - \ h_{t(s_1 \ldots s_j} \, h_{s_{j+1} \ldots s_p)} \, \w^t_{n+1} 
       \right\} \, . 
\end{eqnarray*}
\end{proposition}

\begin{definition*}
We let $h^{(p)} = \{ h_{s_1\cdots s_p}\}$ denote the $p$-th order coefficients; and $h = \cup\, h^{(p)}$ denote the complete set of coefficients.  
\end{definition*}

\begin{remark} \label{R:hinvar}
The $h$ are projective invariants.  To be precise, suppose that $e \in \cF_S$.  Let $A \in \tGL(W)$ and set $\wt S = A(S)$.  Then $\tilde e = L_A(e) \in \cF_{\wt S}$ by Remark \ref{R:LAFS}.  From Proposition \ref{P:h} and the left-invariance \eqref{E:Linvar} of $\w$, we see that $h(e) = h(\tilde e)$.  Thus $h$ is invariant under projective transformation.  In Proposition \ref{P:hcomplete} we show that, for fixed $e \in \cF_S$, the coefficients $h(e)$ characterize an analytic CR-hypersurface up to projective transformation.
\end{remark}

\begin{remark*}  The coefficients $h$ are the CR--analogs of the coefficients of the Fubini forms of a complex subvariety of $\bP W$, cf. \cite[Chapter 3]{MR2003610}.
\end{remark*}

\begin{definition*}In this setting, we say two CR-hypersurfaces $S, \wt S \subset \bP W$ \emph{agree to order $p\ge1$} if there exists a bundle map $\phi : \cF_S \to \cF_{\wt S}$ with the property that $\w^\s_0 = \phi^*\wt\s^\s_0$, $\w^n_0 = \phi^* \wt\w^n_0$ (first-order agreement) and $\phi^*\wt h^{(q)} = h^{(q)}$ for all $2\le q \le p$.  
\end{definition*}

\begin{remark}\label{R:Jensen}
This is a considerably weaker condition than Jensen's notion of $p$-th order projective deformation when $p\ge 2$.  For example, in our case second order agreement is equivalent to $\w^\s_0 = \phi^* \wt \w^\s_0$, $\w^n_0 = \phi^*\wt\w^n_0$ and $\w^{n+1}_s = \phi^* \wt \w^{n+1}_s$.  Jensen's condition that $S$ and $\wt S$ be second-order projective deformations of each other includes the additional relations $\Omega^a_b = \phi^* \wt\O^a_b$, with $1\le a,b\le m$.  Differentiating Jensen's second-order equations yields $\w = \phi^*\wt\w$ which implies that $S$ and $\wt S$ are projectively equivalent.  See \cite{MR836926}.
\end{remark}

\begin{proof}[Proof of Proposition \ref{P:h}]
From \eqref{E:h3} we see that Proposition \ref{P:h} holds for $p=2$.  Differentiating \eqref{E:h3} and applying Cartan's Lemma \ref{L:Cartan4w} yields functions $h_{rstu} : \cF_S\to\bR$, completely symmetric in their indices, such that 
\begin{equation} \label{E:h4} \renewcommand{\arraystretch}{1.3}
\begin{array}{rcl}
  h_{rstu}\,\w^u_0 & = & \td h_{rst} \, + \, 
  h_{rst}\,( 2\,\w^0_0 + \w^{n+1}_{n+1}) \, - \, 
  h_{ust}\,\w^u_r \, - \, h_{rut}\,\w^u_s \, - \, h_{rsu}\,\w^u_t \\
  & & + \, h_{rs}\,\w^0_t \, + \, h_{st}\,\w^0_r \, + \, h_{tr}\,\w^0_s \, - \, 
  ( h_{rs}\,h_{tu} \,+\, h_{st}\,h_{ru} \,+\, h_{tr}\,h_{su} ) \, \w^u_{n+1} \, .
\end{array}\end{equation}
This establishes Proposition \ref{P:h} in the case $p=3$.  The general case may not be established by induction on $p$.  Details are left to the reader.
\end{proof}

%------------------------------------------------------------------------------
\subsection{Local coordinate computation} \label{S:local_comp_w}
%------------------------------------------------------------------------------

In this section we give a local coordinate computation of $h_{st}$ at a point $w_o \in S$.  Given $w_o = [1:0:\cdots:0] \in S$ fix coordinates $z = (z^1,\ldots,z^m) \mapsto [1:z^1:\cdots:z^m]$ on $\bP W$ in a neighborhood of $w_o$ so that $S$ is locally given as a graph
$$
  y^m \ = \ f(z^1,\bar z^1,\ldots,z^{m-1},\bar z^{m-1},x^m)
      \ = \ f(x^1,y^1,\ldots,x^{m-1},y^{m-1},x^m) \, , \quad z^j = x^j + \ti\, y^j \, ,
$$
over its (real) tangent space $T_p S = \{y^m=0\}$.  Fix the index range $0<\a,\b< m$.  Locally the tangent space is spanned by 
$$
  \sfe_{2\a} \ = \ \partial_{x^\a} \ + \ f_{x^\a} \, \partial_{y^m} \,,\quad
  \sfe_{2\a+1} \ = \ \partial_{y^\a} \ + \ f_{y^\a} \, \partial_{y^m} \,,\quad
  \sfe_{2m} \ = \ \partial_{x^m} \ + \ f_{x^m} \, \partial_{y^m} \, .
$$
The vector $\sfe_{2m+1} = \sfJ \sfe_{2m} = \partial_{y^m} - f_{x^m} \partial_{x^m}$
completes the local framing of the tangent space to a framing of $\bC^m$ (as a real vector space).  If we set 
\begin{eqnarray*}
  \sfe_0 & = & \partial_{x^0} \ + \ x^\a \, \partial_{x^\a} 
  \ + \ y^\a \, \partial_{y^\a} 
  \ + \ x^m \, \partial_{x^m} \ + \ f \, \partial_{y^m} \\
  \sfe_1 \ = \ \sfJ \sfe_0 & = & \partial_{y^0} 
  \ - \ y^\a \, \partial_{x^\a} \ + \ x^\a \, \partial_{y^\a} 
  \ + \ x^m \, \partial_{y^m} \ - \ f \, \partial_{x^m}
\end{eqnarray*}
then we may regard $\sfe = (\sfe_0,\ldots,\sfe_{2m+1})$ as framing of $V$.  However this frame does not satisfy \eqref{E:e}.  We modify $\sfe$ as follows.  Set 
$$
  e_0 \ = \ \sfe_0 \, , \quad 
  e_1 \ = \ \sfe_1 \, , \quad 
  e_{2m} \ = \sfe_{2m} \, , \quad 
  e_{2m+1} \ = \ \sfe_{2m+1} \, .
$$
Let $\d = 1/(1 + f_{x^m}{}^2)$ and define 
\begin{eqnarray*}
  e_{2\a} & = & 
  \sfe_{2a} \ + \ \d \, ( f_{y^\a} - f_{x^a} f_{x^m} ) \, \sfe_{2m} \\
  & = & \partial_{x^\a} \ + \ \d \, ( f_{y^\a} - f_{x^\a} f_{x^m}) \partial_{x^m}
  \ + \ \d \, ( f_{x^\a} + f_{y^\a} f_{x^m}) \partial_{y^m} \, ,\\
  e_{2\a+1} & = & 
  \sfe_{2\a+1} \ - \ \d \, ( f_{x^\a} + f_{y^a} f_{x^m} ) \, \sfe_{2m} \\
  & = & \partial_{y^\a} \ - \ \d \, ( f_{x^\a} + f_{y^\a} f_{x^m}) \partial_{x^m}
  \ + \ \d \, ( f_{y^\a} - f_{x^\a} f_{x^m}) \partial_{y^m}\, .
\end{eqnarray*}
Then $e = (e_0,\ldots,e_{2m+1})$ is a local section of $\cF_S$.  The dual coframe is $e^0 = \td x^0$, $e^1 = \td x^1$,
\begin{eqnarray*}
  e^{2a} & = & \td x^a \ - \ x^\a \, \td x^0 \ + \ y^\a \, \td y^0 \, ,\quad
  e^{2a+1} \ = \ \td y^a \ - \ y^\a \, \td x^0 \ - \ x^\a \, \td y^0 \, , \\
  e^{2m} & = & \td x^m \ - \ x^m \, \td x^0 \ + \ f\, \td y^0
  \ - \ \d \, (f_{y^\a} - f_{x^\a} f_{x^m}) \, \td x^\a 
  \ + \ \d \, ( f_{x^\a} + f_{y^\a} f_{x^m}) \, \td y^\a \\
  e^{2m+1} & = & \td y^m \ - \ f \, \td x^0 \ - \ x^m \, \td y^0
  \ - \ \d \, (f_{x^\a} + f_{y^\a} f_{x^m}) \, \td x^\a 
  \ - \ \d \, (f_{y^\a} - f_{x^\a} f_{x^m}) \, \td y^\a \, .
\end{eqnarray*}

Notice that, at $z(w_o) = 0$ we have $f = 0 = \td f$, where the last equality  follows from the fact that the coordinates locally realize $S$ as a graph over its tangent space at $w_o$.  Differentiating at $w_o$ yields
\begin{eqnarray*}
  \td e_0 & = & \td x^\a \, e_{2\a} \, + \, \td y^\a \, e_{2\a+1} 
  \, + \, \td x^m \, e_{n} \, ,\quad
  \td e_1 \, = \, -\td y^\a \, e_{2\a} \, + \, \td x^\a \, e_{2\a+1} 
  \, + \, \td x^m \, e_{n+1}\, , \\
  \td e_{2\a} & = & \td f_{y^\a} \, e_{n} \ + \ \td f_{x^\a} \, e_{n+1}\, , \quad
  \td e_{2\a+1} \ = \ -\td f_{x^\a} \, e_{n} \ + \ \td f_{y^\a} \, e_{n+1}\, , \\
  \td e_{2m} & = & \td f_{x^m} \, e_{n+1}\, , \quad 
  \td e_{2m+1} =\ = \ -\td f_{x^m} \, e_{n} \, . 
\end{eqnarray*}
From \eqref{E:de} we see that, under pulled-back by the section $e$, the only nonzero components $\w^j_k$ of the Maurer-Cartan form are 
$$ \renewcommand{\arraystretch}{1.3}
\begin{array}{llll}
  \w^{2\a}_0 \ = \ \td x^\a \,,\quad & 
  \w^{2\a}_1 \ = \ -\td y^\a \,, \quad & 
  \w^{2\a+1}_0 \ = \ \td y^\a \,, \quad & 
  \w^{2\a+1}_1 \ = \ \td x^\a \,,\\
  \w^n_0 \ = \ \td x^m \,, & 
  \w^n_{2\b} \ = \ \td f_{y^\b} \,, & 
  \w^n_{2\b+1} \ = \ -\td f_{x^\b} \, , & 
  \w^n_{n+1} \ = \ -\td f_{x^m} \,,\\
  \w^{n+1}_1 \ = \ \td x^m \,, & 
  \w^{n+1}_{2\b} \ = \ \td f_{x^\b} \,, & 
  \w^{n+1}_{2\b+1} \ = \ \td f_{y^\b} \,, &
  \w^{n+1}_n \ = \ \td f_{x^m} \, .
\end{array}
$$
Equation \eqref{E:h2} yields (again, at $z(w_o) = 0$)
\begin{equation} \label{E:df=h} \renewcommand{\arraystretch}{1.3}
\begin{array}{rrr}
  h_{2\a,2\b} \ = \ f_{x^\a x^\b} \, , \quad &
  h_{2\a,2\b+1} \ = \ f_{x^\a y^\b} \, , \quad &
  h_{2\a,n} \ = \ f_{x^\a x^m} \, , \\
  h_{2\a+1,2\b+1} \ = \ f_{y^\a y^\b} \, , \quad &
  h_{2\a+1,n} \ = \ f_{y^\a x^m} \, , \quad &
  h_{n,n} \ = \ f_{x^m x^m} \, .
\end{array}\end{equation}
To be precise, the coefficients $h$ on the left are evaluated at $e(w_o) \in \cF_S$, and the derivatives on the right are evaluated at $z(w_o) = 0$.  More generally, for $p > 1$, Proposition \ref{P:h} yields 
$$
  h_{s_1 \ldots s_p t} \, \w^t_0 \ = \
  \td h_{s_1 \ldots s_p} \ - \ p \, h_{t(s_1 \ldots s_{p-1}} \, \w^t_{s_p)} 
  \ - \ \sum_{j=1}^{p-2} \, \tbinom{p}{j} \, 
   h_{n(s_1 \ldots s_j} \, h_{s_{j+1} \ldots s_p)} \, \td f_{x^m} \, . 
$$
In the case that none of the indices are equal to 1, we see that $h_{s_1\cdots s_p t}$ is a $(p+1)$-st derivative of $f$, modulo terms involving lower over derivatives.  This is why we refer to the $h_{s_1 \cdots s_p}$ as $p$-th order invariants.  Moreover, given the $h_{s_1\cdots s_p}$ we may recover the partial derivatives of $f$ at $0$ and therefore the hypersurface $S$ (assuming $S$ is real analytic).  That is, a connected, analytic $S$ is completely determined by $h(e)$.

The frame $e(w_o)$ over $w_o$ is not unique in this respect.  Given second frame $\tilde e$ over $w_o$, it is always possible to find a local section of the form above through $\tilde e$.  

\begin{proposition} \label{P:hcomplete}
Let $S, \wt S \subset \bP W$ be two connected, analytic CR-hypersurfaces.  Then there exists a complex projective linear transformation $A : \bP W \to \bP W$ such that $A(S) = \wt S$ if and only if there exist frames $e \in \cF_{S}$ and $\tilde e \in \cF_{\wt S}$ such that $h(e) = {\wt h}(\tilde e)$.
\end{proposition}

\begin{remark*}
In Section \S\ref{S:h1} we will show that the $h_{s_1\cdots s_p}$ with some $s_i = 1$ are redundant.
\end{remark*}

\begin{proof}
Let $A : \bP W \to \bP W$ be an invertible projective linear transformation, and let $L_A : \tGL(W) \to \tGL(W)$ denote the lift (left multiplication) to $\tGL(W)$.  If $A$ maps $S$ to $\wt S$, then the lift $L_A$ maps $\cF_S$ to $\cF_{\wt S}$.  Given $e \in \cF_S$, let $\tilde e = \wt A(e)$.  Then \eqref{E:Linvar} and Proposition \ref{P:h} yield $h(e) = {h}(\tilde e)$.

Conversely, suppose that $h(e) = h(\tilde e)$ for some frames $e \in \cF_S$ and $\tilde e \in \cF_{\wt S}$.  Define $A \in \tGL(W)$ by $A e = \tilde e$.  Then $A(S)$ and $\wt S$ agree to infinite order at $\tilde e$.  It follows from the discussion preceding the lemma that $A(S) = \wt S$.
\end{proof}

%------------------------------------------------------------------------------
\subsection{Redundancy of the index 1 in \boldmath $h$ \unboldmath} \label{S:h1}
%------------------------------------------------------------------------------

The alert reader will have noticed that \eqref{E:w} allows us to write \eqref{E:dw} as 
$$
  0 \ = \ \w^{n+1}_2 \wedge \w^2_0 \ + \ \cdots \ + \  
          \w^{n+1}_{n-1}\wedge\w^{n-1}_0
          \ + \ (\w^{n+1}_n - \w^1_0) \wedge \w^n_0
$$
Cartan's Lemma then yields a smaller set of second-order invariants $\{ h_{st} \ | \ 2 \le s , t \le n \}$, defined by $\w^{n+1}_\sigma = h_{\sigma t} \,\w^\t_0$ and $\w^{n+1}_n - \w^1_0 = h_{nt}\,\w^t_0$, $2 \le t \le n$.  That is, the coefficients of $h^{(2)}$ with some index equal to 1 are redundant.  

This pattern continues.  First define
$$
  J^{2\a}_{2\b} \ := \ 0 \ =: \ J^{2\a+1}_{2\b+1} \,, \quad
  J^{2\a+1}_{2\b} \ := \ \d^\a_\b \quad \hbox{ and } \quad
  J^{2\a}_{2\b+1} \ := \ -\d^\a_\b \, .
$$
Then $J^\s_\n\,J^\n_\t = -\d^\s_\tau$, and \eqref{E:w} and Corollary \ref{C:P_S} yield
\begin{equation} \label{E:Jh} \renewcommand{\arraystretch}{1.3}
\begin{array}{rclrclrcl}
  \w^\sigma_1 & = & J^\sigma_\tau \, \w^\tau_0 \,,\quad &
  \w^1_\sigma & = & -J_\sigma^\tau\,\w^0_\tau \,,\quad &
  \w^0_\s & = & J^\t_\s\,\w^1_\t \, , \\
  \w^n_\sigma & = & J^\nu_\sigma\,h_{\nu\tau}\,\w^\tau_0 \,,\quad &
  \w^\s_{n+1} & = & J^\s_\tau\,\w^\tau_n \,,\quad &
  \w^\s_\t & = & -J^\s_\m\,\w^\m_\n\,J^\n_\t \,.
  % \quad \hbox{or} \quad \w^\s_\n\,J^\n_\t \ = \ J^\s_\n\,\w^\n_\t \, .
\end{array}\end{equation}
To see that the third-order invariants $h_{rst}$ with some index equal to 1 are redundant differentiate $0 = h_{11} = h_{1\sigma}$ and $h_{1n} = 1$ to obtain
\begin{eqnarray}
  \nonumber
  0 & = & h_{11s}  \ = \ h_{1nn} \,,\\
  \label{E:dh2_id}
  h_{1\s n} & = & -J_\sigma^\tau\, h_{\tau n} \,,\\
  \nonumber
  h_{1\s\t} & =& -J_\s^\n\,h_{\n\t} \ - \ J_\t^\n\,h_{\s\n} \, .
\end{eqnarray}

The general observation may be established by induction:  Define 
$$
  1^{\sfa} \sigma_1\cdots \sigma_p n^{\sfb} \ := \ 
  \underbrace{1\cdots1}_{\sfa \ 1'\mathrm{s}}\sigma_1\cdots\sigma_p
  \underbrace{n\cdots n}_{\sfb \ n'\mathrm{s}} \, .
$$
Assume that there exists $r_o \in \bZ$ such that every $h_{1^\sfa\s_1\cdots\s_p n^\sfb} \in \{h^{(r)} \ | \ r \le r_o \}$ with $\sfa > 0$ is a linear combination of the $\{h_{\tau_1\cdots\tau_p n^\sfb}\}$.  Above we saw that this is the case for $r_o = 2,3$.  Then a computation with Proposition \ref{P:h} that makes use of \eqref{E:w} and \eqref{E:Jh} will show that every $h_{1^\sfa\s_1\cdots\s_p n^\sfb} \in h^{(r_o+1)}$ with $\sfa>0$ is a linear combination of the $\{h_{\tau_1\cdots\tau_p n^\sfb}\}$.  Details are left to the reader.  As a corollary to Lemma \ref{P:hcomplete} we have the following.

\begin{corollary} \label{C:h}
  The $h_o = \{ h_{\sigma_1\cdots\sigma_p n^\sfb} \}$ form a complete set of projective differential invariants for an analytic CR-hypersurface.  That is, two hypersurfaces $S$ and $\wt S$ are projectively equivalent if and only if there exists $e \in \cF_S$ and $\tilde e \in \cF_{\wt S}$ such that $h_o(e) = h_o(\tilde e)$.
\end{corollary}

We close the section by pointing out that, while the $h_{1^\sfa\tau_1\cdots\tau_p n^\sfb}$ with $\sfa>0$ are redundant, they are very convenient for computations; the formulation of Proposition \ref{P:h} would not be so tidy without them.  Indeed we discern no analogous general expression for the differential invariants with respect to $\Omega$.

%------------------------------------------------------------------------------
\section{Projective differential invariants w.r.t. $\Omega$} \label{S:wrtO}
%------------------------------------------------------------------------------

%------------------------------------------------------------------------------
\subsection{Second-order invariants} \label{S:PQ}
%------------------------------------------------------------------------------

By Remark \ref{R:sb} the $\{ \O^\a_m , \barO^\a_m , \O^m_0 \}$ span the semi-basic forms on $\cF_S$.  The $\O$-version of Cartan's Lemma is as follows.

\begin{lemma}[Cartan's Lemma {\cite{MR1083148,MR2003610}} for $\Omega$] \label{L:CartanO}
Suppose there exist complex-valued 1-forms $\zeta_\a$, $\zeta_{\bar\a}$ and $\zeta_m$ on $\cF_S$ such that $\zeta_\a \wedge \O^\a_0 + \zeta_{\bar\a} \wedge \bar\O^\a_0 + \z_m \wedge \O^m_0 = 0$.  Then there exist functions $Z_{\a\b} , Z_{\a\bar\b} , Z_{\bar\a\b} , Z_{\bar\a\bar\b} , Z_\a , Z_{\bar\a} , Z : \cF_S \to \bC$ such that 
\begin{eqnarray*}
  \z_\a & = & Z_{\a\b}\,\O^\b_0 \ + \ Z_{\a\bar\b}\,\bar\O^\b_0 
  \ + \ Z_\a \, \O^m_0 \, , \\
  \z_{\bar\a} & = & Z_{\bar\a\b}\,\O^\b_0 \ + \ Z_{\bar\a\bar\b}\,\bar\O^\b_0 
  \ + \ Z_{\bar\a} \, \O^m_0 \, , \\
  \z_m & = & Z_{\b}\,\O^\b_0 \ + \ Z_{\bar\b}\,\bar\O^\b_0 
  \ + \ Z \, \O^m_0 \, ,
\end{eqnarray*}
and 
$$
  Z_{\a\b} \ = \ Z_{\b\a} \, , \quad
  Z_{\a\bar\b} \ = \ Z_{\bar\b\a} \, , \quad
  Z_{\bar\a\bar\b} \ = \ Z_{\bar\b\bar\a} \, .
$$
\end{lemma}

Recall that \eqref{E:w^n+1_0} holds on $\cF_S$.  Differentiating with the Maurer-Cartan equation \eqref{E:mcC}, we obtain
\begin{equation} \label{E:dOm0}
  0 \ = \ \td \, (\Omega^m_0 - \bar\Omega^m_0) \ = \ 
  ( \Omega^0_0 + \bar\Omega^m_m - \bar\Omega^0_0 - \Omega^m_m ) \wedge \Omega^m_0
  \ - \ \Omega^m_\a \wedge \Omega^\a_0 
  \ + \ \bar\Omega^m_\a \wedge\bar\Omega^\a_0 \, .
\end{equation}
Note that $\Omega^0_0 + \bar\Omega^m_m - \bar\Omega^0_0 - \Omega^m_m$ takes values in $\ti\bR$.  So Lemma \ref{L:CartanO} yields functions $P_{\a\b} = P_{\b\a} \,,\ \bar P_{\a\bar\b} = -P_{\b\bar\a} \,,\ P_{\a m} \,,\ P_{mm} =-\bar P_{mm}: \cF_S \to \bC$ such that
\begin{subequations}  \label{E:1}
\begin{eqnarray}
  \label{E:1b}
  \Omega^m_\a 
  & = & P_{\a\b} \, \Omega^\b_0 \ + \ P_{\a\bar\b} \, \bar\Omega^\b_0 
  \ + \ P_{\a m} \, \Omega^m_0 \, ,\\
  \label{E:1a}
  (\Omega^0_0 \ + \ \bar\Omega^m_m) \ - \ 
  (\bar\Omega^0_0 \ + \ \Omega^m_m) & = & -P_{\a m} \, \Omega^\a_0 
  \ + \ \bar P_{\a m} \, \bar\Omega^\a_0 \ + \ P_{mm}\, \Omega^m_0 \,.
\end{eqnarray} \end{subequations}

Making use of \eqref{E:w}, \eqref{E:W=(w)}, \eqref{E:w=(W)}, \eqref{E:h2} and \eqref{E:1} we see that 
\begin{equation} \label{E:P=(h)} \renewcommand{\arraystretch}{1.3}
\begin{array}{rcl}
  P_{\a\b} & = & \half \,(h_{2\a,2\b+1} + h_{2\a+1,2\b} ) 
  \ + \ \ihalf\,(h_{2\a,2\b} - h_{2\a+1,2\b+1} ),\\
  P_{\a\barb} & = & \half\,(h_{2\a+1,2\b} - h_{2\a,2\b+1})
  \ + \ \ihalf\,(h_{2\a,2\b} + h_{2\a+1,2\b+1}) \,,\\
  P_{\a m} & = & h_{n,2\a+1} \ + \ \ti \, h_{n,2\a} \,,\\
  P_{mm} & = & -2\ti\, h_{nn} \,.
\end{array}\end{equation}
In particular, the coefficients on the right-hand side of \eqref{E:1} are the second-order differential invariants on $\cF_S$.  Indeed, the $-\ihalf P_{\a\barb}$ are the coefficients of the \emph{Levi form} in a local coordinate frame.  See Section \ref{S:local_comp_O}.  (In an abuse of language we will often refer to $P_{\a\barb}$ as the coefficients of the Levi form.)  The frame bundle $\cF_S$ admits a sub-bundle $\cP_S$ on which the $P_{\a m}$ and $P_{mm}$ vanish; see Lemma \ref{L:P_S}.

\begin{remark*}
Together the $P_{\a\b}$ and $P_{\a\barb}$ may be identified with the derivative of a Gauss map (see Section \ref{S:gauss}).  The are, respectively, the anti-hermitian and hermitian parts of the CR second fundamental form. 
\end{remark*}

%------------------------------------------------------------------------------
\subsection{Local coordinate computation} \label{S:local_comp_O}
%------------------------------------------------------------------------------

Here we convert the expressions of \S\ref{S:local_comp_w} to $z$ and $\bar z$ derivatives.  Recall $\del_z = \half ( \del_x - \ti \del_y)$ and $\del_{\bar z} = \half( \del_x + \ti \del_y)$, so that \eqref{E:df=h} and \eqref{E:P=(h)} yield
\begin{eqnarray}
  \nonumber
  f_{z^\a z^\b} & = & \fourth \, ( f_{x^\a x^\b} - f_{y^\a y^\b}) \ - \ 
  \ifourth \, ( f_{x^\a y^\b} + f_{y^\a x^\b}) 
  \ = \ -\ihalf P_{\a\b} \, , \\
  \label{E:levi}
  f_{z^\a \barz^\b} & = & \fourth \, ( f_{x^\a x^\b} + f_{y^\a y^\b} ) \ + \ 
  \ifourth \, ( f_{x^\a y^\b} - f_{y^\a x^\b}) 
  \ = \ - \ihalf \, P_{\a\barb} \, . 
\end{eqnarray}
So we have the power series representation for $y^m = f$ at $p$,
\begin{equation} \label{E:ps} \renewcommand{\arraystretch}{1.3}
\begin{array}{rcl}
  y^m & = & h_{2\a,2\b} \, x^\a x^\b \ + \ 2 \, h_{2\a,2\b+1} \, x^\a y^\b
  \ + \ h_{2\a+1,2\b+1} \, y^\a y^\b \ + \ \Delta \\
  & = & 
  \fourth (h_{2\a,2\b} - h_{2\a+1,2\b+1} - 2 \, \ti \, h_{2\a,2\b+1}) \, z^\a z^\b 
  \\ & &  + \ \half (h_{2\a,2\b} + h_{2\a+1,2\b+1} + 
               \ti\,(h_{2\a,2\b+1} - h_{2\a+1,2\b})) \, z^\a \barz^\b \\
  & & 
  + \ \fourth (h_{2\a,2\b} - h_{2\a+1,2\b+1} + 2 \, \ti \, h_{2\a,2\b+1}) \,
    \barz^\a \barz^\b \ + \ \Delta \\
  & = & -\ihalf P_{\a\b} z^\a z^\b \ - \ \ti\,P_{\a\barb} z^\a \barz^\b 
  \ + \ \ihalf\bar P_{\a\b} \barz^\a \barz^\b  \ + \ \Delta \, ,
\end{array}
\end{equation}
where $\Delta$ denotes the terms of degree greater than two, as well as the quadratic terms involving $x^m$.  

\begin{definition*}
From \eqref{E:levi} we see that $-\ihalf P_{\a\barb}$ is the \emph{Levi form}.  
\end{definition*}

\begin{remark*}
Classically, \emph{strong pseudoconvexity} is the condition that the Levi form be negative definite.  If we replace $e_0$ with $-e_0$, then the coefficients $P_{\a\barb}$ change sign.  So when working on the frame bundle $\cF_S$ we will define $S$ to be strongly pseudoconvex if the coefficients $P_{\a\barb} : \cF_S \to \bC$ define a definite form.  
\end{remark*}

\begin{definition} \label{D:SCLC}
A strongly pseudoconvex hypersurface $S$ is \emph{strongly $\bC$-linearly convex} (SCLC) if \emph{no} real line tangent to $\wh H_zS$ makes second order contact with $\wh S$ at $w \in \hat z$.  
\end{definition}

\noindent Let $\sfz = (\sfz^1,\ldots,\sfz^{m-1}) \in \bC^{m-1}$.  Define $P(\sfz,\sfz) = P_{\a\b}\sfz^\a\sfz^\b$ and $L(\sfz,\bar\sfz) = P_{\a\barb}\sfz^\a\bar\sfz^\b$.  From \eqref{E:ps} we see that the surface $S$ is strongly $\bC$--linearly convex if and only if
\begin{equation} \label{E:SCLC}
  0 \ \not= \ \ihalf P(\sfz,\sfz) \ + \ \ti L(\sfz,\bar\sfz) \ + \ \overline{\ihalf P(\sfz,\sfz)} 
  \ = \ -\tIm P(\sfz,\sfz) + \ti\,L(\sfz,\bar\sfz) 
  %\ihalf 
  %\left( \begin{array}{cc} \sfz^\a & \bar \sfz^\a \end{array}\right) \,
  %\left( \begin{array}{rr} P_{\a\b} & P_{\a\barb} \\ -\bar P_{\a\barb} 
  %& - \bar P_{\a\b} \end{array}\right) \,
  %\left( \begin{array}{c} \sfz^\b \\ \bar \sfz^\b \end{array} \right)
  %\ = \ \tRe(\ti P(z,z)) \ - \ |z|^2 \ = \ -\tIm(P(z,z)) \ - \ |z|^2 
\end{equation}
for all $0 \not= \sfz \in \bC^{m-1}$.  Equivalently, 
\begin{equation} \label{E:SCLC3}
  |\tIm P(\sfz,\sfz)| < |L(\sfz,\bar\sfz)| 
  \quad \forall \ 0 \not= \sfz \in \bC^{m-1} \,.  
\end{equation}
Equation \eqref{E:SCLC} can be expressed as 
\begin{equation} \label{E:SCLC2}
  0 \ \not= \ \ti\,\left( \begin{array}{cc} \sfz^t & \bar\sfz^t \end{array} \right) 
  \left( \begin{array}{cc} P & L \\ L^t & -\bar P \end{array} \right) 
  \left( \begin{array}{c} \sfz \\ \bar\sfz \end{array} \right) \,,
\end{equation}
for all $0 \not= \sfz \in \bC^{m-1}$, where $P = (P_{\a\b})$ and $L = (P_{\a\barb})$.  In particular the matrix above must be invertible.  Indeed, 
\begin{equation} \label{E:SCLC_inv}
  \left( \begin{array}{cc} P & L \\ L^t & -\bar P \end{array} \right)^{-1} \ = \ 
  \left( \begin{array}{cc} Q & M \\ M^t & -\bar Q \end{array} \right)
  \quad \hbox{ where } \quad \renewcommand{\arraystretch}{1.3}
  \begin{array}{l} Q = \overline{L^{-1}P} (P\overline{L^{-1}P} - L)^{-1}\,, \\
                   M = ( \bar P L^{-1} P - \bar L )^{-1} \, .  
  \end{array}
\end{equation}
Note that $Q^t = Q$ and $\bar M^t = -M$.

%------------------------------------------------------------------------------
\subsection{Second order projectively invariant tensors on \boldmath $S$ \unboldmath }\label{S:F2}
%------------------------------------------------------------------------------

In this section we illustrate the frame bundle construction of the second fundamental form of $S$.

Let $J_z : T_z\bP W \to T_z \bP W$ denote the complex structure on $\bP W$.  Given $z \in S$, let $H_z = T_z S \, \cap \, J_z(T_z S)$ denote the maximal complex subspace of $T_zS$.  Analogously define $\wh H_z S =  \wh T_zS \,\cap\, \sfJ(\wh T_zS) \subset W$.  Given $z \in \bP W$, let $\hat z = L_z \in W$ denote the corresponding (complex) line through the origin.  Then 
$$
  H_z S \ = \ L_z^* \ot ( \wh H_z S / L_z ) \, , 
$$
and $HS \to S$ defines a rank $m-1$ complex vector bundle over $S$.  Similarly, define the normal (complex) line bundle $NS \to S$ by 
$$
  N_z S \ = \ T_z \bP W / H_z S \ = \ L_z^* \ot ( W / \wh H_zS) \, .
$$

Given $f = (f_0 , \ldots , f_n) \in \cF_S$, set $z = [f_0] \in S$ and let $(f^0 , \ldots , f^n)$ denote the dual basis of $W^*$.  Then $\sff_\a = f^0 \ot ( f_\a \ \tmod \ L_z)$ defines a basis of $H_zS$, and $\sff_m = f^0 \ot ( f_m \ \tmod \ \wh H_z S)$ spans $N_zS$.  Let $(\sff^1,\ldots, \sff^{m-1})$ denote the dual basis of $H^*_zS = (H_zS)^*$.  Define 
\begin{eqnarray*}
  \mathbf{P}_f & = & P_{\a\b}\,\sff^\a\sff^\b\ot\sff_m 
  \ \in \ (\tSym^2 H^*_z S) \ot N_zS \,,
  \quad \hbox{ and } \\
  \mathbf{L}_f & = & P_{\a\barb}\,\sff^\a \bar\sff^\b \ot \sff_m 
  \ \in \ H^*_z S \,\ot\, \bar H^*_zS \,\ot\, N_zS \, .
\end{eqnarray*}
We claim that $\mathbf{P}$ and $\mathbf{L}$ descend to well-defined sections of $(\tSym^2 H^*) \ot NS$ and $H^*S \ot \bar H^*S \ot NS$, respectively, over $S$.  (The tensors $\mathbf{P}$ and $\mathbf{L}$ are respectively the anti-hermitian and hermitian parts of the second fundamental form of $S$.)  To establish the claim it suffices to show that $\mathbf{P}$ and $\mathbf{L}$ are constant on fibres of $\cF_S$. This is seen as follows.

First consider a change of frame (\S\ref{S:tw}) of the form 
\begin{equation} \label{E:block}
  \tilde f_0 \ = \ g^0_0 \, f_0 \, , \quad 
  \tilde f_\a \ = \ g_\a^\b \, f_\b \, , \quad 
  \tilde f_m \ = \ g^m_m \, f_m \, .  
\end{equation}
Such a change of frame is called a \emph{block fibre motion} on $\cF_S$.  
%Since $\cF_S \subset \cF \simeq \tSL(W)$ it is necessary that $\tdet(g)=g^0_0\,\tdet(g^\a_\b)\,g^m_m = 1$.  
Computing with \eqref{E:tO} and \eqref{E:1b} we see that the change in $P$ and $L$ is given by 
\begin{equation} \label{E:tQ}
  \wt{P}_{\a\b} \ = \
  \frac{P_{\c\e} \, g^\c_\a \, g^\e_\b}{g^0_0 \, g^m_m} 
  \quad \hbox{ and } \quad
  \wt{P}_{\a\barb} \ = \ 
  \frac{P_{\c\bare} \, g^\c_\a \, \bar g^\e_\b}{\bar g^0_0 \, g^m_m} \, .
\end{equation}
The transformation in the coefficients $P_{\a\b}$ and $P_{\a\barb}$ precisely cancels transformation in $\sff$ so that $\mathbf{P}_f = \mathbf{P}_{\tilde f}$ and $\mathbf{L}_f = \mathbf{L}_{\tilde f}$.

Next, consider a change of frame (\S\ref{S:tw}) of the form
\begin{equation} \label{E:shear}
  \tilde f_0 \ = \ f_0 \, , \quad 
  \tilde f_\a \ = \ f_\a \ + \ g^0_\a \, f_0 \, , \quad
  \tilde f_m  \ = \ f_m \ + \ g_m^0 \, f_0 \ + \ g^\a_m \, f_\a \, .
\end{equation} 
These changes of frame are \emph{shear fibre motions}.  The entire group of fibre motions on $\cF_S$ is generated by block and shear transformations.  Computing with \eqref{E:tO} and \eqref{E:1b} we see that $P$ and $L$ are unchanged by shear fibre motions: 
\begin{equation} \label{E:PLshear}
  \wt{P}_{\a\b} \ = \ P_{\a\b}
  \quad \hbox{ and } \quad 
  \wt{P}_{\a\barb} \ = \ P_{\a\b} \, .
\end{equation}
As a consequence of \eqref{E:tQ} and \eqref{E:PLshear} we see that $\mathbf{P}_f$ and $\mathbf{L}_f$ are constant under the fibre motions.  Our claim follows.  

The tensors $\mathbf{P}$ and $\mathbf{L}$ are projectively invariant.  To be precise, suppose that $\wt S$ is a second hypersurface that is projectively equivalent to $S$ via $A \in \tGL(W)$.  Note that $A$ naturally identifies $(\tSym^2 H_z^*S) \ot N_zS$ with $(\tSym^2H_{Az}^*\wt S) \ot N_{Az}\wt S$, and $H_z^*S\,\ot\,\bar H_{z}^* S\,\ot\,N_z S$ with $H^*_{Az}\wt S\,\ot\,\bar H^*_{Az} \wt S\,\ot\,N_{Az}\wt S$.  Thus we may define the pull-backs $A^* \mathbf{\wt{P}}$ and $A^*\mathbf{\wt L}$.  From \eqref{E:LinvarO} we have $P = \wt{P} \circ A$ and $Q = \wt{Q} \circ A$.  Thus $A^* \wt{\mathbf{P}} = \mathbf{P}$ and $A^*\mathbf{\wt L} = \mathbf{L}$.  Whence projective equivalence.  

\begin{remark*}
G. Jensen \cite{MR739891} proved that a nondegenerate hypersurface $S$ has $\mathbf{P} \equiv 0$ if and only if $S$ is projectively equivalent to a quadric hypersurface (Section \ref{S:hyperquadric}).  In related work Detraz and Tr\'epreu \cite{MR1094711} showed that the hyperquadric appears as one of two types of hypersurfaces in $\bC^n$ that are characterized by an elliptic system. 
\end{remark*}

%\begin{remark} \label{R:P=0}
%From \eqref{E:tQ} and \eqref{E:PLshear} we see that, when restricted to a fibre of $\pi : \cF_S \to S$, $P_{\a\b}$ is either identically zero or no where vanishing.
%\end{remark}

\begin{remark*}
In the case that $S$ is strongly $\bC$-linearly convex the inverse matrix \eqref{E:SCLC_inv} exists.  If $Q = (Q^{\a\b})$ and $M = (Q^{\a\barb})$, then 
\begin{eqnarray*}
  \mathbf{Q}_f & = & Q^{\a\b}\,\sff_\a\sff_\b\ot\sff^m 
  \ \in \ (\tSym^2 H_z S) \ot N^*_zS \,,
  \quad \hbox{ and } \\
  \mathbf{M}_f & = & -Q^{\a\barb}\,\sff_\a \bar\sff_\b \ot \sff^m 
  \ \in \ H_z S \,\ot\, \bar H_zS \,\ot\, N^*_zS \, .
\end{eqnarray*}
similarly define projectively invariant tensors $\mathbf{Q}$ and $\mathbf{M}$ on $S$.  Here $\sff^m \in N^*_zS$ is dual to $\sff_m \in N_zS$.  As we will see in \eqref{E:II*}, $\mathbf{Q}$ and $\mathbf{M}$ are pull-backs (under a `lifted' Gauss map) of the second fundamental form of the dual-hypersurface.
\end{remark*}

%------------------------------------------------------------------------------
\section{A reduction of the bundle $\cF_S$}\label{S:P_S}
%------------------------------------------------------------------------------

In this section we will show that the second order coefficients $h_{\sigma n}$ and $h_{nn}$ (equivalently, for $P_{\a m}$ and $P_{mm}$) may be normalized to zero.  That is, $\cF_S$ admits a sub-bundle $\cP_S$ of adapted frames over $S$ on which the coefficients vanish.

%------------------------------------------------------------------------------
\subsection{Normalizations}\label{S:norms}
%------------------------------------------------------------------------------

\begin{lemma} \label{L:P_S}
Let $S \subset \bP W$ be a CR-hypersurface.  There exists a sub-bundle $\cP_S \subset \cF_S$ on which the coefficients $P_{\a m}$ and $P_{mm}$ vanish.  In particular the restriction of the Maurer-Cartan form $\Omega$ to $\cP_S$ satisfies 
\begin{subequations} \label{SE:normP}
\begin{eqnarray}
  \label{E:P_S0}
  \barO^m_0 & = & \O^m_0 \\
  \label{E:P_S1}
   \Omega^m_\a & = & P_{\a\b} \, \Omega^\b_0
   \ + \ P_{\a\barb}\,\bar\Omega^\a_0\, ,\\
  \label{E:P_S2}
  \Omega^0_0 \ + \ \bar\Omega^m_m & = & \bar\Omega^0_0 \ + \ \Omega^m_m \, ,
\end{eqnarray} 
with $P_{\b\a} = P_{\a\b}$ and $\bar P_{\a\barb} = -P_{\b\bara}$.
\end{subequations}
\end{lemma}

\begin{proof}
Since \eqref{E:w^n+1_0} holds on $\cF_S$, equation \eqref{E:P_S0} is immediate.  By \eqref{E:1} it remains to show that $P_{\a m}$ and $P_{mm}$ can be normalized to zero.  Consider a change of frame (\S\ref{S:tw}) of the form \eqref{E:shear}.  Set $h^0_m = g^0_\a\,g^\a_m - g^0_m = (g^{-1})^0_m$.  Then \eqref{E:tO} yields
$$ \begin{array}{rclrcl}
  \wt\O^0_0 & = & \O^0_0 \ - \ g^0_\a \,\O^\a_0 \ + \ h^0_m\,\O^m_0 \,,\quad &
  \wt\Omega^\a_0 & = & \Omega^\a_0 \ - g^\a_m \, \Omega^m_0 \,,\quad
  \wt\Omega^m_0 \ = \ \Omega^m_0 \,,\\
  \wt\Omega^m_\a & = & \Omega^m_\a \ + \ g^0_\a \, \Omega^m_0 \,,\quad &
  \wt\O^m_m & = & \O^m_m \ + \ g^\a_m\,\O^m_\a \ + \ g^0_m\,\O^m_0 \, .
\end{array} $$
The coefficients $P_{\a m}$ transform as $\wt{P}_{\a m} = P_{\a m} + g^0_\a + P_{\a\b} \, g^\b_m + P_{\a\barb}\,\bar g^\b_m$.  By selecting $g^0_\a$ appropriately we may may normalize to 
\begin{equation} \label{E:Pam=0}
  P_{\a m} \ = \ 0  \, .
\end{equation}
Similarly, \eqref{E:1a} yields $\wt{P}_{mm} = P_{mm} + (g^0_\a\, g^\a_m - \bar g^0_\a\,\bar g^\a_m ) - 2\,(g^0_m - \bar g^0_m)$.  So we may use $g^0_m$ to normalize $\tIm(P_{mm}) = 0$.  Since $P_{mm}$ takes values in $\ti\bR$, we have 
\begin{equation} \label{E:Pmm=0}
  P_{mm} \ = \ 0 \, .
\end{equation}
Let $\cP_S \subset \cF_S$ denote the sub-bundle on which \eqref{E:Pam=0} and \eqref{E:Pmm=0} hold.
\end{proof}

\begin{corollary} \label{C:P_S}
The second-order invariants $h_{sn}$, $2 \le s \le n$, vanish on $\cP_S$.  In particular, the following relations hold on $\cP_S$:
$$ 
%\renewcommand{\arraystretch}{1.3}
%\begin{array}{ll}
  0 \ = \ \w^{n+1}_0 \ = \ \w^n_1 \,, \quad
  \w^{n+1}_1 \ = \ \w^n_0 \,,\quad
  \w^{n+1}_\sigma \ = \ h_{\sigma\tau}\,\w^\tau_0 \,, \quad 
  \w^{n+1}_n \ = \ \w^1_0 \ = \ -\w^n_{n+1} \, . 
%\end{array} 
$$
\end{corollary}

\begin{proof}
The corollary follows from \eqref{E:w}, \eqref{E:w^n+1_0}, \eqref{E:h2_id}, \eqref{E:h2}, \eqref{E:P=(h)} and Lemma \ref{L:P_S}.
\end{proof}

\begin{remark} \label{R:norms}
The bundle $\cP_S$ is of (real) dimension $2 m^2 + 3$ and is preserved by the shear fibre motions \eqref{E:shear} satisfying
\begin{eqnarray}
  \label{E:normg0}
  g^0_\a \ + \ P_{\a\b} \, g^\b_m \ + \ P_{\a\barb}\,\bar g^\b_m & = & 0 
  \,, \qquad \hbox{ and }\\
  \nonumber
  4\ti\,\tIm(g^0_m) \ = \ 2\,(g^0_m - \bar g^0_m) & = & 
  (g^0_\a\, g^\a_m \ - \ \bar g^0_\a\,\bar g^\a_m ) \ = \ 
  2\ti\,\tIm(g^0_\a\,g^\a_m) \\
  \nonumber
  & = & -P_{\a\b}\,g^\a_m\,g^\b_m \ + \ \overline{P_{\a\b}\,g^\a_m\,g^\b_m}
  \ - \ 2\,P_{\a\barb}\,g^\a_m\,\bar g^\b_m \, .
\end{eqnarray}
\end{remark}

\begin{lemma} \label{L:P_Spreserved}
Let $S \subset \bP W$ be a CR-hypersurface and $A \in \tGL(W)$.  Then $L_A(\cP_S) = \cP_{\wt S}$, where $\wt S = A(S) \subset \bP W$.
\end{lemma}

\begin{proof}
In Remark \ref{R:LAFS} we observed that $L_A(\cF_S) = \cF_{\wt S}$.  Thus $L_A(\cP_S) \subset \cF_{\wt S}$.  From left-invariance \eqref{E:LinvarO} of $\Omega$ we see that $\eqref{SE:normP}$ holds on $L_A(\cP_S)$.  Hence $L_A(\cP_S) = \cP_{\wt S}$.
\end{proof}

\begin{proposition} \label{P:proj_equivP}
Two CR-hypersurfaces $S , \wt{S} \subset \bP W$ are projectively equivalent if and only if there exists a smooth map $\phi : \cP_S \to \cP_{\wt S}$ such that $\phi^*(\Omega_{|\cP_{\wt S}}) = \Omega_{|\cP_S}$.  (Equivalently, $\phi^*(\omega_{|\cP_{\wt S}}) = \omega_{|\cP_S}$.)
\end{proposition}

\noindent The proof of Proposition \ref{P:proj_equivP} is identical to that of Proposition \ref{P:proj_equiv}.

%------------------------------------------------------------------------------
\subsection{Example: the hyperquadric} \label{S:hyperquadric}
%------------------------------------------------------------------------------

The homogenous model of a strongly pseudoconvex hypersurface in $\bC\bP^n$ is the hyperquadric.  Fix a frame $\sfe = (\sfe_0,\ldots, \sfe_{n+1})$ of $V$ in $\cF$ and let $\sff=(\sff_0,\ldots,\sff_{m})$ be the corresponding basis of $W$.  Define linear coordinates $z = z^a \sff_a$ on $W$.  Let $[z] = [z^0: \cdots:z^{m}]$ be the corresponding complex homogeneous coordinates on $\bP W$.  Define 
$$
  q(z) \ = \ \bar z^t Q z \ = \ 
  \ti ( z^0 \bar z^m - z^m \bar z^0 ) \ + \ \sum_\a z^\a \bar z^\a \, .
$$
Let $\tSU(1,m) = \tSU(W,q) \subset \tGL(W)$ be the subgroup stabilizing $q$.  Then the Lie algebra $\fsu(1,m)$ is given by matrices $X = (X^a_b) \in \fgl(W)$ satisfying the following:
\begin{equation} \label{E:su}
  \renewcommand{\arraystretch}{1.3}
  \begin{array}{l}
   0 \ = \ \bar X^0_0 + X^m_m \ = \  \bar X^m_0 - X^m_0 
     \ = \ \bar X^0_m - X^0_m \ = \ \bar X^m_\a - \ti X^\a_0 
     \ = \ \bar X^\a_m + \ti X^0_\a \, , \\
   (X^\a_\b) \in \mathfrak{u}(m-1) \, , \quad \hbox{ and } \quad 
   \mathrm{Tr}(X) = 0 \, .
  \end{array} 
\end{equation}

Let $S = \{ q=0 \} \subset \bP W$.  Note that $\sfe \in \cF_S$.  Define $\cG = \tSU(1,m) \cdot \sfe \subset \cF_S$.  Then $\cG$ is a sub-bundle of the adapted frames over $S$, and is naturally identified with $\tSU(1,m)$.  The Maurer-Cartan form, when restricted to $\cG$, takes values in $\fsu(1,m)$.  In particular, 
$$
  \Omega^m_\a \ = - \ti \, \bar\Omega^\a_0 
  \qquad \hbox{ and } \qquad 
  \barO^0_0 \ + \ \O^m_m \ = \ 0 \, .
$$
We see that $\cG \subset \cP_S$, and $P_{\a\b} = 0$  and $P_{\a\barb} = -\ti\d_{\a\b}$ on $\cG$.  

%------------------------------------------------------------------------------
\subsection{An \boldmath $\w$--coframe on $\cP_S$ \unboldmath}
%------------------------------------------------------------------------------

By Corollary \ref{C:P_S} we have 
$0 = h_{\sigma n} = h_{nn}$ on the sub-bundle $\cP_S$.  Differentiating these expressions and applying \eqref{E:h2_id}, \eqref{E:h3}, \eqref{E:dh2_id} and \eqref{E:w} yields $0 = h_{1sn}$ and
\begin{subequations} \label{E:w0}
\begin{eqnarray}
  \label{E:w1s}
  \w^1_\sigma \ + \ h_{\sigma\tau}\,\w^\tau_n & = & 
  -h_{\sigma\tau n}\,\w^\tau_0 \ - \ h_{\sigma nn}\,\w^n_0 \,,\\
  \label{E:w0s}  
  \w^0_\sigma \ + \ J_\sigma^\nu\,h_{\nu\tau}\,\w^\tau_n & = & 
  -J^\nu_\sigma\,h_{\nu n \tau}\,\w^\tau_0 
  \ - \ J^\nu_\sigma\,h_{\nu nn}\,\w^n_0 \,,\\
  \label{E:w1n}
  2 \, \w^0_{n+1} \ \stackrel{\eqref{E:w}}{=} \ -2\,\w^1_n 
  & = & h_{nn\tau}\,\w^\tau_0 \ + \ h_{nnn}\,\w^n_0 \, .
\end{eqnarray}\end{subequations}
Note that \eqref{E:w0s} is a consequence of \eqref{E:Jh} and \eqref{E:w1s}.

\begin{lemma} \label{L:PSframe}
The 
$$
  E_\w(\cP_S) \ := \ \{ \w^0_0 ,\, \w^t_0 ,\, \w^{2\a}_{2\b} ,\, \w^{2\a+1}_{2\b} 
  ,\, \w^0_n ,\, \w^\tau_n ,\, \w^n_n \}
$$
form a coframing of $\cP_S$.  The remaining components of $\w$ are given by \eqref{E:w}, Corollary \ref{C:P_S}, \eqref{E:w1s} and \eqref{E:w1n}.
\end{lemma}

\noindent The claim that $E_\w(\cP_S)$ is a coframing on $\cP_S$ follows by dimension count.  See Remark \ref{R:norms}.  

\begin{corollary}  \label{C:PSframe}
Two CR-hypersurfaces $S , \wt S \in \bP W$ are projectively equivalent if and only if there exists a smooth map $\phi : \cP_{S} \to \cP_{\wt S}$ such that 
\begin{equation} \label{E:phi*1}
  \phi^*E_\w(\cP_{\wt S}) \ = \ E_\w(\cP_S) \,,
\end{equation}
and 
\begin{equation} \label{E:phi*2} \renewcommand{\arraystretch}{1.3}
\begin{array}{rcl}
   h_{\s\t} & = & {\wt h}_{\s\t} \circ \phi \, ,\quad
   h_{\s\t n} \  = \ {\wt h}_{\s\t n} \circ \phi \,,\\
   h_{\s nn} & = & {\wt h}_{\s nn} \circ \phi \,,\quad
   h_{nnn} \ = \ {\wt h}_{nnn} \circ \phi \, .
\end{array}\end{equation}
Equivalently, $\phi^* (\w_{|\cP_{\wt S}}) = \w_{|\cP_S}$.
\end{corollary}

\begin{proof}
By the relations \eqref{E:w^n+1_0}, \eqref{E:h2}, \eqref{E:w0s} and \eqref{E:w1n} the equations \eqref{E:phi*1} and \eqref{E:phi*2} hold if and only if $\phi^*(\w_{|\cP_{\wt S}}) = \w_{|\cP_S}$.  The corollary then follows from  Proposition \ref{P:proj_equivP}.
\end{proof}

In order to establish the $\O$--versions of Lemma \ref{L:PSframe} and Corollary \ref{C:PSframe} (in Section \ref{S:PSframeO}) we must first compute two derivatives in Sections \ref{S:dPS1} and \ref{S:dPS2}.

%------------------------------------------------------------------------------
\subsection{Differentiate \eqref{E:P_S1}} \label{S:dPS1}
%------------------------------------------------------------------------------

Differentiating \eqref{E:P_S1} with \eqref{E:mcC} produces
$$
  0 \ = \ \cP_{\a\b}\wedge \O^\b_0 \ + \ \cP_{\a\barb} \wedge \bar\O^\b_0 
  \ - \ \big(\Omega^0_\a \, + \, P_{\a\b}\,\O^\b_m \, + \, P_{\a\barb}\,\barO^\b_m
  \big) \wedge \O^m_0 \, ,
$$
where
\begin{eqnarray*}
  \cP_{\a\b} & = & \td P_{\a\b} \ + \ P_{\a\b}\,(\O^0_0 + \O^m_m) 
  \ - \ P_{\c\b}\,\O^\c_\a \ - \   P_{\a\c}\,\O^\c_\b  \,,\\
  \cP_{\a\barb} & = & \td P_{\a\barb} \ + \ P_{\a\barb}\,(\barO^0_0 + \O^m_m)
   \ - \ P_{\c\barb}\,\O^\c_\a \ - \ P_{\a\barc}\,\barO^\c_\b\,.
\end{eqnarray*}
Lemma \ref{L:P_S} implies 
\begin{equation} \label{E:cP2id}
  \cP_{\a\b} \ = \ \cP_{\b\a} \quad \hbox{ and } \quad 
  \bar\cP_{\a\barb} \ = \ -\cP_{\b\bara} \, .
\end{equation}
Cartan's Lemma \ref{L:CartanO} and \eqref{E:cP2id} yield functions $P_{\a\b\c} ,\, P_{\a\b\barc},\, P_{\a\b m} ,\, P_{\a\barb\c} ,\, P_{\a\barb\barc} ,\, P_{\a\barb m} ,\, P_{\a mm} : \cP_S \to \bC$ such that 
%\begin{subequations} \label{E:3rd_a}
\begin{eqnarray} 
  \nonumber %\label{E:cPab} 
  \cP_{\a\b} & = & P_{\a\b\c}\,\O^\c_0 \ + \ P_{\a\b\barc}\,\barO^\c_0
  \ + \ P_{\a\b m}\,\O^m_0 \, ,\\
  \nonumber %\label{E:cPabarb}
  \cP_{\a\barb} & = & P_{\a\barb\c}\,\O^\c_0 \ + \ P_{\a\barb\barc}\,\barO^\c_0 
  \ + \ P_{\a\barb m}\,\O^m_0 \,,\\
  \label{E:O0a}
  \Omega^0_\a \ + \ P_{\a\b}\,\O^\b_m \ + \ P_{\a\barb}\,\barO^\b_m & = & 
  -P_{\a\b m}\,\O^\b_0 \ - \ P_{\a\barb m}\,\barO^\b_0 \ - \ P_{\a mm}\,\O^m_0 \, , 
\end{eqnarray}
%\end{subequations}
with 
\begin{equation} \nonumber %\label{E:Pid} 
\renewcommand{\arraystretch}{1.3}
\begin{array}{rcl}
  P_{\a\b\c} \ = \ P_{\b\a\c} \ = \ P_{\a\c\b} \,, & 
  P_{\a\b\barc} \ = \ P_{\b\a\barc} \ = \ P_{\a\barc\b} \,,&
  P_{\a\barb\barc} \ = \ P_{\a\barc\barb} \,,\\
  P_{\a\b m} \ = \ P_{\b\a m}\, , &
  \bar P_{\a\barb\c} \ = \ -P_{\b\bara\barc} \,, &
  \bar P_{\a\barb m} \ = \ -P_{\b\bara m} \, .
\end{array} \end{equation}
Note that \eqref{E:O0a} is the $\Omega$--version of \eqref{E:w0s}, and it is straight forward to check that 
\begin{eqnarray*}
  P_{\a\b m} & = & \half \, \left( h_{2\a,2\b+1,n} \,-\, h_{2\b,2\a+1,n} \right)
  \ - \ \ihalf \, \left( h_{2\a,2\b,n} \,-\, h_{2\a+1,2\b+1,n} \right) \,,\\
  P_{\a\barb m} & = & \half\,\left( h_{2\a,2\b+1,n} \,-\, h_{2\b,2\a+1} \right) 
  \ - \ \ihalf\,\left( h_{2\a,2\b,n} \,+\, h_{2\a+1,2\b+1,n} \right) \,,\\
  P_{\a mm} & = & h_{2\a+1,nn} \ + \ \ti\,h_{2\a,nn} \, .
\end{eqnarray*}

%------------------------------------------------------------------------------
\subsection{Differentiate \eqref{E:P_S2}} \label{S:dPS2}
%------------------------------------------------------------------------------

Differentiating \eqref{E:P_S2} produces
\begin{eqnarray*}
  0 & = & \big( -\O^0_\a \, - \, P_{\a\b}\,\O^\b_m \, - \, P_{\a\barb}\,
                 \bar\O^\b_m \big) 
  \wedge \O^\a_0 \ + \ 
  \big( \bar\O^0_\a \, + \, \bar P_{\a\b}\,\bar\O^\b_m \, + \, 
        \bar P_{\a\barb}\,\O^\b_m \big)
  \wedge \bar\O^\a_0 \\
  & & + \ 
  2\, \big( \bar\O^0_m \,-\,\O^0_m \big) \wedge \O^m_0 \, .
\end{eqnarray*}
Cartan's Lemma \ref{L:CartanO} yields $P_{m^3} : \cP_S \to \bC$ such that 
\begin{equation} \label{E:O0m}
  -2 \, (\O^0_m \,-\,\barO^0_m) \ = \ 
  P_{\a mm} \,\O^\a_0 \ - \ \bar P_{\a mm} \, \bar\O^\a_0 
  \ - \ P_{m^3}\,\O^m_0 \quad \hbox{ with } \quad \bar P_{m^3} \ = \ -P_{m^3} \, . 
\end{equation}
This expression is the $\Omega$--version of \eqref{E:w1n}.  It is straight-forward to check that 
$$
  P_{mmm} \ = \ -2\ti\,h_{nnn}\, .
$$

\begin{remark*}
The coefficient functions $P_{\sfa\sfb\sfc}$ in this section and Section \ref{S:dPS1} are the third-order invariants of $S$ with respect to $\Omega$.  Unlike $P_{\a\b}$ and $P_{\a\barb}$ they do not yield well-defined tensors on $S$.
\end{remark*}

%------------------------------------------------------------------------------
\subsection{A \boldmath $\O$--coframe on $\cP_S$ \unboldmath} \label{S:PSframeO}
%------------------------------------------------------------------------------

In analogy with Lemma \ref{L:PSframe} we have the following.

\begin{lemma} \label{L:PSframeO}
The 1-forms 
$$ 
E_\O(\cP_S) \ := \ \{ \O^0_0 \,, \barO^0_0 \,, \O^\a_0 ,\, \barO^\a_0 \,, \O^m_0 ,\, \O^\a_\b ,\, \barO^\a_\b \,, \O^m_m ,\, \O^\a_m ,\, \barO^\a_m ,\, \O^0_m  \}
$$
form a coframing of $\cP_S$ (over $\bR$) and the remaining components of $\Omega$ are given by \eqref{SE:normP}, \eqref{E:O0a} and \eqref{E:O0m}.
\end{lemma}

\noindent The lemma follows by dimension count, see Remark \ref{R:norms}.

\begin{remark*}
In the case that $S$ is strongly $\bC$-linearly convex, the $\{ \O^\a_0,\,\O^\a_m\}$ in the coframing may be replaced with the $\{ \O^m_\a ,\, \O^0_\a \}$.  See Section \ref{S:SCLC}.
\end{remark*}

\begin{corollary} \label{C:PSframeO}
Two CR-hypersurfaces $S , \wt S \in \bP W$ are projectively equivalent if and only if there exists a smooth map $\phi : \cP_S \to \cP_{\wt S}$ such that 
$$
  \phi^* E_\O(\cP_{\wt S})  \ =  \ E_\Omega(\cP_S) 
$$
and 
$$ \renewcommand{\arraystretch}{1.3}
\begin{array}{rcl}
   P_{\a\b} & = & \wt{P}_{\a\b} \circ \phi \, ,\quad
   P_{\a\barb} \ = \ \wt{P}_{\a\barb} \circ \phi \,, \\
   P_{\a\b m} & = & \wt{P}_{\a\b m} \circ \phi \,, \quad
   P_{\a\barb m} \  = \ \wt{P}_{\a\barb m} \circ \phi \,,\\
   P_{\a mm} & = & \wt{P}_{\a mm} \circ \phi \,,\quad
   P_{mmm} \ = \ \wt{P}_{mmm} \circ \phi \, .
\end{array}$$
Equivalently, $\phi^*(\Omega_{|\cP_{\wt S}}) = \Omega_{|\cP_S}$.
\end{corollary}

\noindent The proof is identical to that of Corollary \ref{C:PSframe}, and is left to the reader.

%------------------------------------------------------------------------------
\subsection{The case that \boldmath $S$ \unboldmath is SCLC} \label{S:SCLC}
%------------------------------------------------------------------------------

In the case that $S$ is strongly $\bC$--linearly convex (Definition \ref{D:SCLC}) equations \eqref{E:P_S1}, \eqref{E:O0a} and \eqref{E:O0m} have alternate formulations.  Let $Q = (Q^{\a\b})$ and $M = (Q^{\a\barb})$ be given by \eqref{E:SCLC_inv}.  Then 
\begin{subequations} \label{SE:SCLC_alt}
\begin{eqnarray}
  \label{E:Oa0}
  \O^\a_0 & = & Q^{\a\b}\,\O^m_\b \ - \ Q^{\a\barb}\,\barO^m_\b \,,\\
  \label{E:Oam}
  \O^\a_m \ + \ Q^{\a\b}\,\O^0_\b \ - \ Q^{\a\barb}\,\barO^0_\b & = & 
  -Q^{\a\b m}\,\O^m_\b \ + \ Q^{\a\barb m}\,\barO^m_\b 
  \ - \ Q^{\a mm}\,\O^m_0 \,, \\
  \label{E:O0m_alt}
    -2 \, (\O^0_m \,-\,\barO^0_m) & = &
  Q^{\a mm} \,\O^m_\a \ - \ \bar Q^{\a mm} \, \bar\O^m_\a 
  \ - \ P_{m^3}\,\O^m_0 \,,
\end{eqnarray}
\end{subequations}
where 
\begin{eqnarray*}
  Q^{\a\b m} & = & Q^{\a\c}\,P_{\c\e m}\,Q^{\e\b} \ + \ 
  Q^{\a\c}\,P_{\c\bare m}\,Q^{\b\bare} \ - \ 
  Q^{\a\barc}\,\bar P_{\c\e m}\,Q^{\b\bare} \ + \ 
  Q^{\a\barc}\,P_{\e\barc m} \, Q^{\e\b} \,,\\
  Q^{\a\barb m} & = & Q^{\a\c}\,P_{\c\e m}\,Q^{\e\barb} \ - \ 
  Q^{\a\c}\,P_{\c\bare m}\,\bar Q^{\e\b} \ + \ 
  Q^{\a\barc}\,\bar P_{\c\e m}\,\bar Q^{\e\b} \ + \ 
  Q^{\a\barc}\,P_{\e\barc m} \, Q^{\e\barb} \,,\\
  Q^{\a mm} & = & Q^{\a\c}\,P_{\c mm} \ - \ Q^{\a\barc}\, \bar P_{\c mm} \, .
\end{eqnarray*}
As we will see in Section \ref{S:SD}, the coefficients $Q$ above may be identified with the $P$ coefficients of the dual hypersurface.

%------------------------------------------------------------------------------
\section{Dual hypersurfaces} \label{S:SD}
%------------------------------------------------------------------------------

In this section we define the Gauss map of a CR-hypersurface $S \subset \bP W$ and characterize the self-dual strongly $\bC$--linearly convex (SCLC) hypersurfaces (Theorem \ref{T:SD}).

%------------------------------------------------------------------------------
\subsection{Gauss map} \label{S:gauss}
%------------------------------------------------------------------------------
Let $\cF_S$ be the adapted frame bundle over a SCLC hypersurface $S$.  Given $f = (f_0,\ldots,f_m) \in \cF_S$, let $z = [f_0] \in S$, and let $\wh H_z S= \wh T_zS \,\cap \, \sfJ(\wh T_zS)$ be the maximal complex subspace of $\wh T_zS \subset W$.  Note that $\wh H_z S= \tspan\{ f_0 , \ldots , f_{m-1} \}$ (see \S\ref{S:F_S}).  Let $(f^0 , \ldots , f^m) \in W^*$ be the basis dual to $f$.  Note that $f^m$ vanishes when restricted to $\wh H$.  Since $\wh H$ depends only on $z \in S$, this implies that, modulo rescaling, $f^m$ depends only on $z \in S$.  In particular, the map $\cF_S \to \bP W^*$ sending $f \mapsto [f^m] \in \bP W^*$ descends to $S$ where it defines the \emph{Gauss map}
$$
  \gamma : S \ \to \ \bP W^* \, .
$$
Under the identification of $\bP W^*$ with the Grassmannian $\tGr(m,m+1)$ of $m$-dimensional $\bC$--planes in $W$,  $\gamma$ is the map sending $z \in S$ to $\wh H_zS \in \tGr(m,m+1)$.

\begin{definition*}
The image $S^* = \gamma(S)$ is the \emph{dual of $S$.}   
\end{definition*}

Let $\z = (\z_0,\ldots,\z_m) \in \tGL(W^*)$ denote a basis of $W^*$.  The Maurer-Cartan form $\Lambda$ on $\tGL(W^*)$ is defined by $\td \z_a = \Lambda_a^b\,\z_b$.  Define a map $\Gamma : \tGL(W) \to \tGL(W^*)$ by $\Gamma(f) = (f^m,f^\a,f^0)$.  Note that $\Gamma^2 = \tId$.  From \eqref{E:df} we see that 
\begin{equation} \label{E:df^m}
  \td f^a \ = \ -\Omega^a_b \, f^b \, .
\end{equation}   
Thus,
\begin{equation} \label{E:GammaPB}
  \Gamma^* \left( \begin{array}{ccc}
    \Lambda^0_0 & \Lambda^0_\b & \Lambda^0_m \\
    \Lambda^\a_0 & \Lambda^\a_\b & \Lambda^\a_m \\
    \Lambda^m_0 & \Lambda^m_\b & \Lambda^m_m
  \end{array} \right) \ = \ 
  - \left( \begin{array}{ccc}
  \O^m_m & \O^\b_m & \O^0_m \\
  \O^m_\a & \O^\b_\a & \O^0_\a \\
  \O^m_0 & \O^\b_0 & \O^0_0
  \end{array} \right) \, .
\end{equation}
The following lemma is well-known; see \cite[\S2.5]{MR2060426}.  We give a proof as a warm-up to Theorem \ref{T:SD}.

\begin{lemma} \label{L:*}
If $S \subset \bP W$ is a strongly $\bC$--linearly convex hypersurface, then $S^* \subset \bP W^*$ is also a strongly $\bC$--linearly convex  hypersurface and $\Gamma(\cP_S) = \cP_{S^*}$.
\end{lemma}

\begin{proof}
From \eqref{E:df^m} we have 
\begin{equation} \label{E:d_gamma}
  \td f^m \ = \ -\Omega^m_0 \,f^0 - \Omega^m_\b \, f^\b - \Omega^m_m \, f^m \, .
\end{equation}
It is a consequence of strong $\bC$--linear convexity that $\{ \O^m_0 , \O^m_\a , \barO^m_\a , \O^m_m , \barO^m_m\}$ are linearly independent over $\bR$ on $\cF_S$, see \eqref{E:Oa0}.  Thus, $\wh T_{\c(z)} S^* = \td \gamma (T_f \cF_S)$ is of real dimension $n+1 = 2m+1$.  In particular, $S^*$ is a hypersurface in $\bP W^*$. 

From \eqref{E:df^m} we see that $\Gamma$ maps $\cF_S$ to $\cF_{S^*}$.  If we restrict $\Gamma$ to $\cP_S$, then \eqref{E:GammaPB} and Lemma \ref{L:P_S} yield
$$\renewcommand{\arraystretch}{1.3}
\begin{array}{rcl}
  \Gamma^*(\bar \Lambda^m_0) & = & \Gamma^*(\Lambda^m_0) \,,\\
  \Gamma^*(\Lambda^m_\a) & = & -\O^\a_0 
  \ = \ -Q^{\a\b}\,\O^m_\b \ + \ Q^{\a\barb}\,\barO^m_\b 
  \ = \ Q^{\a\b}\,\Gamma^*(\Lambda^\b_0) 
  \ - \ Q^{\a\barb}\,\Gamma^*(\bar\Lambda^\b_0)\,, \\
  \Gamma^*( \bar\Lambda^0_0 \,+\, \Lambda^m_m ) & = & 
  \Gamma^*(\Lambda^0_0 \,+\, \bar\Lambda^m_m ) \, .
\end{array}$$
The coefficients $Q = (Q^{\a\b})$ and $M = (Q^{\a\barb})$ above are defined by \eqref{E:SCLC_inv}.  Since $\Gamma : \tGL(W) \to \tGL(W^*)$ is a diffeomorphism, we have 
\begin{equation} \label{E:Lneg} \renewcommand{\arraystretch}{1.3}
\begin{array}{rcl}
  \bar \Lambda^m_0 & = & \Lambda^m_0 \,,\\
  \Lambda^m_\a & = & Q^{\a\b}\,\Lambda^\b_0 
  \ - \ Q^{\a\barb}\,\bar\Lambda^\b_0\,, \\
  \bar\Lambda^0_0 \,+\, \Lambda^m_m & = & 
  \Lambda^0_0 \,+\, \bar\Lambda^m_m \, .
\end{array} \end{equation}
This implies that $\Gamma(\cP_S) = \cP_{S^*}$, and 
\begin{equation} \label{E:II*}
   Q^{\a\b} \ = \ \Gamma^*(P_{\a\b}^*) \quad\hbox{ and } \quad 
   -Q^{\a\barb} \ = \ \Gamma^*(P_{\a\barb}^*)\,,
\end{equation} 
where $P_{\a\b}^*$ and $P_{\a\barb}^*$ are the coefficients of the second fundamental form on $S^*$.

To see that $S^*$ is strongly $\bC$--linearly convex it suffices, by \eqref{E:II*} and \eqref{E:SCLC2}, to show that 
\begin{equation} \label{E:SCLC*}
   0 \ \not= \ \ti\,\left( \begin{array}{cc} \sfw^t & \bar\sfw^t \end{array} \right) 
  \left( \begin{array}{cc} Q & -M \\ \bar M & -\bar Q \end{array} \right) 
  \left( \begin{array}{c} \sfw \\ \bar\sfw \end{array} \right) \,,
\end{equation}
for all $0 \not= \sfw \in \bC^{m-1}$.  Define $0 \not=\sfz\in\bC^{m-1}$ by
$$
  \left( \begin{array}{c} \sfw \\ \bar\sfw \end{array} \right) \ = \ 
  \left( \begin{array}{cc} P & -L \\ \bar L & -\bar P \end{array} \right) 
  \left( \begin{array}{c} \sfz \\ \bar\sfz \end{array} \right) \, .
$$
Then the right-hand side of \eqref{E:SCLC*} is equal to
$$
  \ti\,\left( \begin{array}{cc} \sfz^t & \bar\sfz^t \end{array} \right) 
  \left( \begin{array}{cc} P & -L \\ \bar L & -\bar P \end{array} \right) 
  \left( \begin{array}{c} \sfz \\ \bar\sfz \end{array} \right) \ = \ 
  2\,\tIm \, P(\sfz,\sfz) \ - \ 2\,\ti L(\sfz,\bar\sfz) \, .
$$
Since $S$ is SCLC this quantity is nonzero for all $0 \not= \sfz \in \bC^{m-1}$; see \eqref{E:SCLC3}.  We conclude that \eqref{E:SCLC*} holds and $S^*$ is SCLC.
\end{proof}

%------------------------------------------------------------------------------
\subsection{Self-dual hypersurfaces} \label{S:selfdual}
%------------------------------------------------------------------------------

\begin{definition*}
The SCLC hypersurface $S$ is \emph{self-dual} if there exists an (complex) linear isomorphism $A : W^* \to W$ such that $A(S^*) = S$.  
\end{definition*}

\begin{theorem} \label{T:SD}
A strongly $\bC$--linearly convex CR-hypersurface $S \subset \bP W$ is self-dual if and only if there exists a smooth map $\phi : \cP_S \to \cP_S$ such that 
\begin{equation} \label{E:SDphi*}
  \phi^* \left( \begin{array}{ccc}
    \O^0_0 & \O^0_\b & \O^0_m \\
    \O^\a_0 & \O^\a_\b & \O^\a_m \\
    \O^m_0 & \O^m_\b & \O^m_m
  \end{array} \right)\ = \ 
  - \left( \begin{array}{ccc}
  \O^m_m & \O^\b_m & \O^0_m \\
  \O^m_\a & \O^\b_\a & \O^0_\a \\
  \O^m_0 & \O^\b_0 & \O^0_0
  \end{array} \right) \, .
\end{equation}
In particular, the map $\phi$ must satisfy
\begin{subequations} \label{SE:Q=P}
\begin{eqnarray}
   \label{E:Q=Pa}
   Q^{\a\b} & = & P_{\a\b} \circ \phi \, ,\quad
   -Q^{\a\barb} \ = \ P_{\a\barb} \circ \phi \,, \\
   \label{E:Q=Pb}
   Q^{\a\b m} & = & P_{\a\b m} \circ \phi \,, \quad
   -Q^{\a\barb m} \  = \ P_{\a\barb m} \circ \phi \,,\\
   \label{E:Q=Pc}
   Q^{\a mm} & = & P_{\a mm} \circ \phi \,,\quad
   P_{mmm} \ = \ P_{mmm} \circ \phi \, .
\end{eqnarray} \end{subequations}
\end{theorem}

\begin{proof}
Suppose that $S$ is self dual.  Then there exists a linear isomorphism $A : W^* \to W$ such that $A(S^*) = S$.  Given $\z = (\z_0,\z_\a,\z_m) \in \tGL(W^*)$ define $f = A(\z) = (f_0,f_\a,f_m) \in \tGL(W)$ by $f_a := A(\z_a)$.  This defines an induced map $\mathcal{L}_A : \tGL(W^*) \to \tGL(W)$.  We have 
\begin{equation} \label{E:A*O=L}
  \mathcal{L}_A^*(\Omega_f) = \Lambda_\z \,.  
\end{equation}
This, together with \eqref{E:GammaPB}, implies that $\phi = \cL_A \circ \Gamma$ satisfies \eqref{E:SDphi*}.

Conversely suppose that there exists a smooth map $\phi : \cP_S \to \cP_S$ satisfying \eqref{E:SDphi*}.  Note that the right-hand side of \eqref{E:SDphi*} is a $\fgl(W)$--valued 1-form satisfying the Maurer-Cartan equation.  It now follows from Theorem \ref{T:Cartan} that $\phi = \cL_A \circ \Gamma$ for some linear isomorphism $A : W^* \to W$.

It remains to establish \eqref{SE:Q=P}. The first line \eqref{E:Q=Pa} is a consequence of \eqref{E:II*}.  To establish \eqref{E:Q=Pb} consider the $S^*$ version of \eqref{E:O0a}.  We have 
$$\renewcommand{\arraystretch}{1.5}
\begin{array}{rcl}
\multicolumn{3}{l}{
  \Gamma^*(P^*_{\a\b m})\,\Omega^m_\b \ + \  
  \Gamma^*(P^*_{\a\barb m})\,\bar\O^m_\b
  \ + \ \Gamma^*(P^*_{\a mm})\,\O^m_0 \hbox{\hspace{100pt}} 
  } \\
  \hbox{\hspace{80pt}} & \stackrel{\eqref{E:GammaPB}}{=} & 
    -\Gamma^*(P^*_{\a\b m}\Lambda^\b_0 \,+\, P^*_{\a\barb m}\bar\Lambda^\b_0
  \,+\, P^*_{\a mm}\Lambda^m_0 ) \\
  & \stackrel{\eqref{E:O0a}}{=} &
  \Gamma^*( \Lambda^0_\a \,+\, P^*_{\a\b}\,\Lambda^\b_m \,+\, P^*_{\a\barb}\,
  \bar\Lambda^\b_m ) \\
  & \stackrel{\eqref{E:GammaPB},\eqref{E:II*}}{=} &
  -\O^\a_m \ - \ Q^{\a\b}\,\O^0_\b \ + \ Q^{\a\barb}\,\barO^0_\b \\
  & \stackrel{\eqref{E:Oam}}{=} & 
  Q^{\a\b m}\,\O^m_\b \ + \ Q^{\a\barb m}\,\barO^m_\b \ + \ 
  Q^{\a mm}\,\O^m_0 \, .
\end{array} $$
Thus 
\begin{equation} \label{E:P3*1}
  \Gamma^*(P^*_{\a\b m}) \ = \ Q^{\a\b m} \,, \quad
  \Gamma^*(P^*_{\a\barb m}) \ = \ -Q^{\a\barb m} \quad\hbox{ and }\quad
  \Gamma^*(P^*_{\a mm}) \ = \ Q^{\a mm} \, .
\end{equation}
This yields \eqref{E:Q=Pb}.

To establish \eqref{E:Q=Pc} consider the $S^*$ version of \eqref{E:O0m}
$$\renewcommand{\arraystretch}{1.5}
\begin{array}{rcl}
\multicolumn{3}{l}{
  \Gamma^*(P^*_{\a mm})\,\Omega^m_\a \ - \  
  \Gamma^*(\bar P^*_{\a mm})\,\bar\O^m_\a
  \ - \ \Gamma^*(P^*_{mmm})\,\O^m_0 \hbox{\hspace{100pt}} 
  } \\
  \hbox{\hspace{80pt}} & \stackrel{\eqref{E:GammaPB}}{=} & 
    -\Gamma^*(P^*_{\a mm}\Lambda^\a_0 \,-\, \bar P^*_{\a mm}\bar\Lambda^\a_0
  \,-\, P^*_{mmm}\Lambda^m_0 ) \\
  & \stackrel{\eqref{E:O0m}}{=} &
  2\, \Gamma^*( \Lambda^0_m \,-\, \bar\Lambda^0_m ) 
  \ \stackrel{\eqref{E:GammaPB}}{=} \ -2\,(\O^0_m \,-\,\barO^0_m )\\
  & \stackrel{\eqref{E:O0m_alt}}{=} &
  Q^{\a mm}\,\O^m_\a \ - \ \bar Q^{\a mm}\,\barO^m_\a \ - \ P_{m^3}\,\O^m_0 \, .
\end{array} $$
In particular, 
\begin{equation} \label{E:P3*2}
  \Gamma^*(P^*_{m^3}) \ = \ P_{m^3} \, .
\end{equation}
This yields \eqref{E:Q=Pc}.
\end{proof}

\begin{remark*}
Note that \eqref{E:SDphi*} implies that $(\phi^2)^* \Omega = \Omega$.  Thus there exists $A \in \tGL(W)$ such that $\phi^2 = {L_A}_{|\cP_S}$.
\end{remark*}

\begin{example}
It is well-known that the hyperquadric (see Section \ref{S:hyperquadric}) is self-dual.  Following the notation of Section \ref{S:hyperquadric}, by an argument analogous to the proof of Theorem \ref{T:SD}, the self-duality of the hyperquadric is equivalent to the existence of a map $\phi : \cG \to \cG$ such that \eqref{E:SDphi*} holds.  Given $f = g \cdot \sff \in \cG$, with $g \in \tSU(1,m-1)$, define $\bar f = \bar g \cdot \sff$.  Then $\phi(f_0,f_\a,f_m) = (-\bar f_0 , \ti \bar f_\a , \bar f_m)$ defines a map $\cG \to \cG$ satisfying \eqref{E:SDphi*}.
\end{example}

%------------------------------------------------------------------------------
\subsection{The \boldmath $\w$--version \unboldmath } \label{S:selfdual_w}
%------------------------------------------------------------------------------

The equations \eqref{SE:Q=P} provide second and third-order conditions for a SCLC hypersurface to be self-dual.  If we shift from the $\Omega$ perspective to the $\w$ perspective we obtain $p$-th order conditions as follows.

The equations \eqref{E:W=(w)}, \eqref{E:P=(h)} and \eqref{E:Oa0} imply that we may solve \eqref{E:h2} for $\w^s_0$; that is, there exist functions $k^{st} = k^{ts} : \cF_S \to \bR$ such that $\w^s_0 = k^{st} \, \w^{n+1}_t$.  More generally, in analogy with Proposition \ref{P:h}, there exist functions $k^{s_1\cdots s_p t} : \cF_S \to \bR$, $p>1$, fully symmetric in their indices, and inductively defined by 
\begin{eqnarray*}
  k^{s_1 \ldots s_p t} \, \w_t^{n+1} & = &
  -\td k^{s_1 \ldots s_p} \, + \, (p-1) \, k^{s_1 \ldots s_p} \, \w^{n+1}_{n+1}
  \, + \, k^{s_1 \ldots s_p} \, \w^0_0 \hfill \\
  & &
  + \ p \,  
  \left\{ (p-2) \, k^{(s_1 \ldots s_{p-1}} \w^{s_p)}_{n+1} \,  
       - \, k^{t(s_1 \ldots s_{p-1}} \, \w_t^{s_p)}  \right\} \\
  & & 
  + \ \sum_{j=1}^{p-2} \, \tbinom{p}{j} \, 
      \left\{
      (j-1)\, k^{(s_1 \ldots s_j} \, k^{s_{j+1} \ldots s_p)} \, \w^0_{n+1}
      \ - \ k^{t(s_1 \ldots s_j} \, k^{s_{j+1} \ldots s_p)} \, \w^0_t 
       \right\} \, . 
\end{eqnarray*}
(Our convention is that $k^s = 0$.)

Let $\eta$ denote the Maurer-Cartan form on $\tGL(V^*,\sfJ)$.  It is straight-forward to check that the $\w$--version of \eqref{E:GammaPB} is 
$$
  \Gamma^*\eta^j_k \ = \ -\w^{\n(k)}_{\n(j)} \,,
$$ 
where $\n$ is the permutation of $\{1,\ldots,n\}$ defined by $\n(0) = n+1$, $\n(1) = n$, $\n(2\a) = 2\a+1$ and $\n^2=\tId$.  In particular, if $h^*$ denotes the differential invariants on $S^*$ given by Proposition \ref{P:h}, then
\begin{equation} \nonumber \label{E:PBh*}
  \Gamma^*(h^*_{s_1\cdots s_p}) \ = \ h^*_{s_1\cdots s_p} \circ \Gamma 
  \ = \ k^{\nu(s_1)\cdots \nu(s_p)} \ =: \ k^{s_1\cdots s_p}_\n \, .
\end{equation}
This equation generalizes (with respect to $\w$) the equations \eqref{E:II*}, \eqref{E:P3*1} and \eqref{E:P3*2}.  In particular, by working with respect to $\w$ we may strengthen Theorem \ref{T:SD} to the following.

\begin{theorem} \label{T:SDh}
A strongly $\bC$--linearly convex CR-hypersurface $S \subset \bP W$ is self-dual if and only if there exists a smooth map $\phi : \cF_S \to \cF_S$ such that 
$$
  \phi^*\w^j_k = -\w^{\n(k)}_{\n(j)} \, .
$$
In particular, the map $\phi$ must satisfy
$$
  \phi^*(h_{s_1\cdots s_p}) \ = \ k^{s_1\cdots s_p}_\n \, .
$$
\end{theorem}

\noindent By Proposition \ref{P:hcomplete} we have the following.

\begin{theorem}  \label{T:SD_h}
A strongly $\bC$--linearly convex analytic CR-hypersurface $S$ is self-dual if and only if there exist frames $e , \tilde e \in \cF_S$ such that $h(\tilde e) = k_\n(e)$. 
\end{theorem}

%------------------------------------------------------------------------------
\bibliography{refs.bib}
\bibliographystyle{plain}
%------------------------------------------------------------------------------
\end{document}